\documentclass[final,3p]{elsarticle}

\usepackage{lineno}
\usepackage[colorlinks=true,breaklinks=true,pdftex]{hyperref}
\modulolinenumbers[1]

\usepackage{amssymb,amsmath,amsthm}
\usepackage{bm}
\usepackage{graphicx}
\usepackage{mwe}
\usepackage{algorithmicx}
\usepackage{algpseudocode}
\usepackage{graphics,graphicx,subfigure,float}
\usepackage{booktabs, multirow}
\usepackage[usenames, dvipsnames]{xcolor}
\usepackage[ruled, linesnumbered]{algorithm2e}
\usepackage[colorlinks=true, breaklinks=true]{hyperref}

\usepackage{algpseudocode}
\usepackage{cleveref}
\usepackage[symbol]{footmisc}

\usepackage{tikz,tikz-3dplot}
\usetikzlibrary{intersections,fadings,decorations.pathreplacing} 
\usetikzlibrary{quotes,angles,shapes.arrows}
\usetikzlibrary{arrows,calc}

\usepackage{pgfplots}
\pgfplotsset{compat=1.16}
\usetikzlibrary{shadings}

\usepackage{pxfonts}

\newtheorem{lemma}{Lemma}

\newtheorem{proposition}[lemma]{Proposition}
\newtheorem{corollary}[lemma]{Corollary}
\newcommand{\xqedhere}[2]{%
\rlap{\hbox to#1{\hfil\llap{\ensuremath{#2}}}}}
\graphicspath{{figures/}}

\begin{document}

\begin{frontmatter}

\title{Boundary parameter matching for isogeometric analysis using Schwarz–Christoffel mapping}

\author[inst1]{Ye Ji \corref{cor}}
\ead{y.ji-1@tudelft.nl}
\cortext[cor]{Corresponding author}
\author[inst1]{Matthias M\"oller}
\ead{m.moller@tudelft.nl}
\author[inst2]{Yingying Yu}
\ead{yuyydl@lnnu.edu.cn}
\author[inst3]{Chungang Zhu}
\ead{cgzhu@dlut.edu.cn}
\address[inst1]{Delft University of Technology, Delft Institute of Applied Mathematics, Delft, 2628CD, the Netherlands.}
\address[inst2]{School of Mathematics, Liaoning Normal University, Dalian 116029, China.}
\address[inst3]{School of Mathematical Sciences, Dalian University of Technology, Dalian, 116024, China}

\begin{abstract}
Isogeometric analysis has brought a paradigm shift in integrating computational simulations with geometric designs across engineering disciplines. This technique necessitates analysis-suitable parameterization of physical domains to fully harness the synergy between Computer-Aided Design and Computer-Aided Engineering analyses. The existing methods often fix boundary parameters, leading to challenges in elongated geometries such as fluid channels and tubular reactors. This paper presents an innovative solution for the boundary parameter matching problem, specifically designed for analysis-suitable parameterizations. We employ a sophisticated Schwarz-Christoffel mapping technique, which is instrumental in computing boundary correspondences. A refined boundary curve reparameterization process complements this. Our dual-strategy approach maintains the geometric exactness and continuity of input physical domains, overcoming limitations often encountered with the existing reparameterization techniques. By employing our proposed boundary parameter method, we show that even a simple linear interpolation approach can effectively construct a satisfactory analysis-suitable parameterization. Our methodology offers significant improvements over traditional practices, enabling the generation of analysis-suitable and geometrically precise models, which is crucial for ensuring accurate simulation results. Numerical experiments show the capacity of the proposed method to enhance the quality and reliability of isogeometric analysis workflows.
\end{abstract}


\end{frontmatter}


\section{Introduction}
\label{sec1:introduction}

Isogeometric analysis (IGA) represents a revolutionary development in the integration of Computer-Aided Design (CAD) and Computer-Aided Engineering (CAE), offering a transformative approach across various engineering disciplines \cite{hughes2005isogeometric, cottrell2009isogeometric}. IGA distinguishes itself from the conventional finite element method by utilizing consistent spline-based basis functions for both geometric modeling and numerical simulation. This seamless integration circumvents the need for converting spline-based CAD models into linear mesh models, thereby preserving geometric integrity and eliminating potential errors in the analysis phase.

A pivotal first step in the IGA workflow is to construct an analysis-suitable spline-based parameterization $\bm{x}(\bm{\xi})$ from the Boundary Representation (B-Rep) of the physical domain $\Omega$ \cite{cohen2010analysis}. Research indicates that the quality of this parameterization significantly influences the accuracy and efficiency of subsequent analyses \cite{xu2013optimal, pilgerstorfer2014bounding}. For an analysis-suitable parameterization, ensuring bijectivity is of utmost importance, followed by the minimization of angle and area distortions. Over the past decade, various approaches have been developed \cite{xu2011parameterization, xu2014high, hinz2018elliptic, pan2019isogeometric, liu2020simultaneous, ji2023improved, pan2023g1}.

Many established methods in this field assume that input boundary curves remain unchanged. However, the parameterization speed of these curves plays a crucial role in the quality of the resulting parameterization. It is particularly true for elongated and thin physical domains, as illustrated in Figure~\ref{fig1a:chordLength}. A common challenge arises from designers focusing primarily on the immediate curves, often neglecting the parameterization speed of their opposite counterparts. This oversight frequently leads to a manual determination of boundary correspondences, a practice that limits both automation and efficiency in the parameterization process. Our work addresses this gap, proposing a method that enhances automation and streamlines the parameterization process.

The quality of boundary correspondence has been recognized as a pivotal factor in analysis-suitable parameterization, directly influencing the quality of domain parameterization and then the accuracy of subsequent analyses. Zheng et al. pioneered a method leveraging optimal mass transportation theory to enhance boundary correspondence, primarily focusing on selecting the four corner points \cite{zheng2019boundary}. Zhan et al. utilized a deep neural network approach for corner point selection in planar parameterizations. Both of these methodologies typically commence with selecting four corner points for planar parameterization, followed by reliance on conventional chord-length parameterization to determine the speed of boundary parameters. This approach, however, presents a notable limitation: it can constrain the overall quality and accuracy of the parameterization, which is clearly illustrated in Figure~\ref{fig1a:chordLength}.

\begin{figure}[H]
\centering
\subfigure[Conventional chord-length parameterization]{\includegraphics[width=0.45\linewidth]{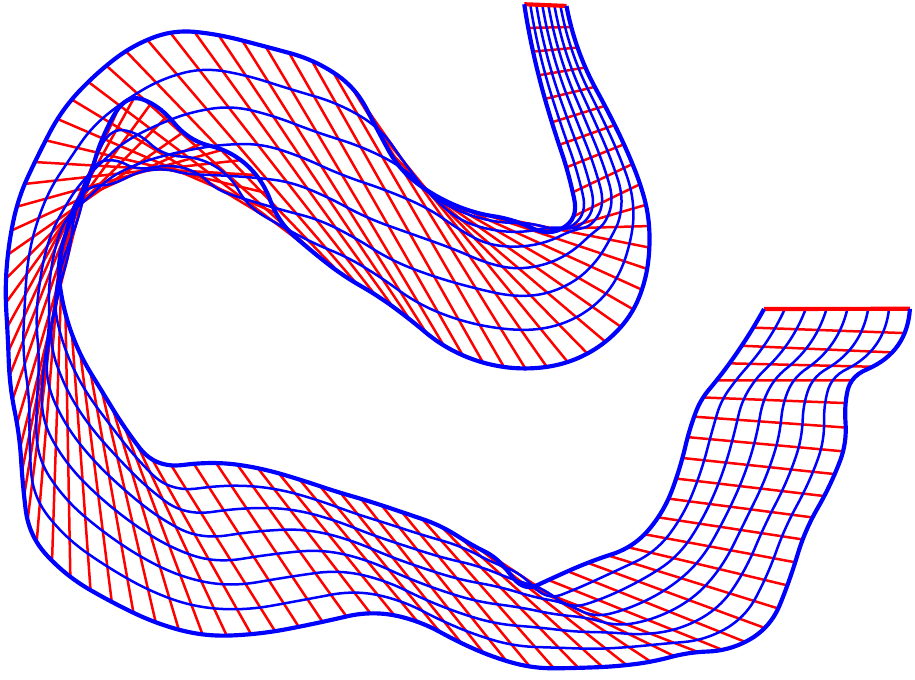}
\label{fig1a:chordLength}}
\quad
\subfigure[Proposed boundary parameter matching method]{\includegraphics[width=0.45\linewidth]{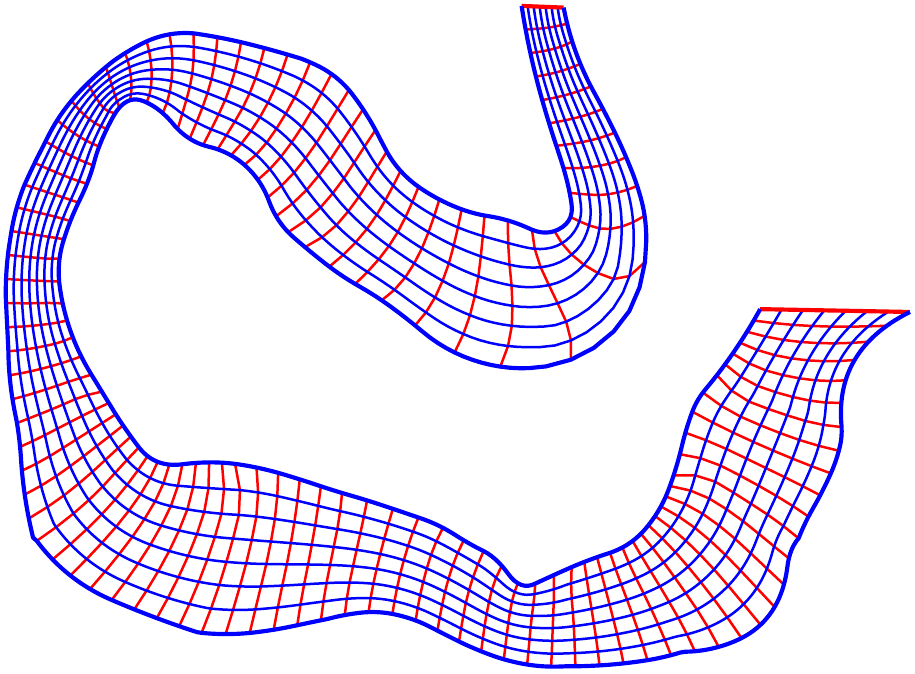}
\label{fig1b:ourResult}}
\caption{Boundary parameter matching problem (Continental Shelf of the Gulf of Mexico): (a) Utilizing conventional chord-length parameterization for the boundary curves, (b) Applying the proposed Schwarz-Christoffel mapping approach for enhanced boundary correspondence.}
\label{fig1:boundary_correspondence_comparison}
\end{figure}

This paper presents a novel approach to address the boundary parameter matching problem for elongated physical domains in isogeometric analysis-suitable parameterization, focusing specifically on the parameter speed of boundary curves. Our method employs the Schwarz-Christoffel (SC) mapping, a specialized form of conformal mapping that has not been previously applied in the IGA framework. This method facilitates precise boundary correspondence, as demonstrated in Figure~\ref{fig1b:ourResult}. To the best of our knowledge, this study is the first to tackle the boundary parameter matching issue in IGA.

The main contributions of this paper are as follows:
\begin{itemize}
    \item Schwarz-Christoffel (SC) mapping is employed to compute markers on input boundary curves. The proposed approach transforms complex NURBS-represented boundary curves into simplified polygons, followed by computing a SC mapping from these polygons to a unit disk. The existence and uniqueness of the SC mapping are underpinned by the well-known Riemann mapping theorem.
    \item A boundary reparameterization scheme is developed, which actively recalibrates parameters to achieve a harmonious alignment between boundary curves. This scheme is specifically designed to maintain the geometric accuracy of the original boundaries. The theoretical proof is given to demonstrate that our proposed scheme preserves the exact geometry of the input boundaries.
    \item Our findings suggest that satisfactory parameterization results can be achieved with a straightforward linear interpolation-based method, offering an alternative to more advanced analysis-suitable parameterization techniques. This is a significant insight, as it implies that simpler methods, when applied effectively, can still yield high-quality parameterizations. The effectiveness of this approach is further validated by numerical experiments, which demonstrate notable improvements in parameterization quality using the proposed method.
\end{itemize}

The remainder of the paper is organized as follows. Section~\ref{sec2:related work} provides a comprehensive survey of existing literature on analysis-suitable parameterization and SC mapping. In Section~\ref{sec3:key ideas}, we define the specific problem our study addresses and introduce the key ideas underlying our proposed method. Section~\ref{sec4:method} presents an in-depth exploration of our methodology. The results and comparative analysis of different methods are discussed in Section~\ref{sec5:experiments}, demonstrating the efficacy of our approach. Section~\ref{sec6:applications} explores the practical applications of our proposed method in solid modeling and IGA simulations. This paper concludes in Section~\ref{sec7:conclusions} with a summary of our key findings and insights into future research directions.

\section{Related work}
\label{sec2:related work}

This section provides a comprehensive review of current methodologies in parameterization as applied in IGA. Next, we will explore the diverse applications and theoretical advancements of SC mapping.

\subsection{Review for analysis-suitable parameterization in IGA}
\label{sec201:review_ASParam}

The importance of parameterization quality in enhancing the accuracy and efficiency of downstream analyses has been emphasized in various studies \cite{cohen2010analysis, xu2013optimal, pilgerstorfer2014bounding}. Algebraic methods, such as Coons patches \cite{farin1999discrete} and Spring patches \cite{gravesen2012planar}, are commonly employed for constructing analysis-suitable parameterizations due to their simplicity and efficiency. However, these methods can sometimes result in non-bijective parameterizations, particularly for complex computational domains.

To address these limitations, a range of methods has been developed over the past decade. Xu et al. \cite{xu2011parameterization} pioneered a nonlinear constrained optimization framework, marking a significant advancement in this field. However, this method often involves a large number of constraints, leading to substantial computational demands. To mitigate this, Wang and Qian \cite{wang2014optimization} and Pan et al. \cite{pan2020volumetric} introduced strategies that effectively reduce computational burdens. Utilizing barrier-type objective functions, Ji et al. efficiently eliminated foldovers through solving a simple unconstrained optimization problem \cite{ji2021constructing}. An alternative approach involves the use of penalty functions and Jacobian regularization techniques, tracing back to Garanzha et al.'s work in mesh untangling \cite{garanzha1999regularization, garanzha2021foldover}. Wang and Ma \cite{wang2021smooth} adopted this strategy, thereby avoiding additional foldover elimination steps. Ji et al. improved upon this by introducing a new penalty term that minimizes numerical errors from previous iterations \cite{ji2022penalty}. Beyond the barrier-type objective function, other studies have delved into quasi-conformal theories, such as Teichm\"{u}ller mapping \cite{nian2016planar} and the low-rank quasi-conformal method \cite{pan2018low}. Additionally, Martin et al. utilized discrete harmonic functions for trivariate B-spline solids \cite{martin2009volumetric}, while Nguyen and J\"uttler \cite{nguyen2010parameterization} and Xu et al. \cite{xu2013constructing} explored sequences and variational methods of harmonic mapping, respectively. Falini et al. approached the problem by computing harmonic maps from physical domains to parametric domains using boundary element methods, followed by inverse mapping approximations via least-squares fitting \cite{falini2015planar}. Building on the principles of Elliptic Grid Generation (EGG), Hinz et al. proposed PDE-based methods that excel in domains with extreme aspect ratios due to their robust convergence properties \cite{hinz2018elliptic, hinz2020pde}. Further advancing this domain, Ji et al. introduced an enhanced elliptic PDE-based parameterization technique, notable for its speed and ability to produce uniform elements near concave/convex boundary regions in general domains \cite{ji2023improved}. Zhang et al. \cite{zhang2012solid, zhang2013conformal} and Liu et al. \cite{liu2014volumetric} developed volumetric spline-based parameterizations for IGA, utilizing T-splines known for their local refinement capabilities. In addition, Xu et al. \cite{xu2019efficient} and Ji et al. \cite{ji2022curvature} introduced $r$-adaptive methods that focus on minimizing curvature metrics.


The aforementioned studies presuppose the availability of boundary representations and maintain these input boundary curves and surfaces in their original fixed form. Establishing a coherent correspondence between distinct shapes is a fundamental challenge in shape analysis, various methods have been established during past decades \cite{van2011survey, sahilliouglu2020recent}. However, these traditional methods fall short in the context of our research as the underlying assumptions are not applicable. Liu et al. allow the parameterization of boundary curves may be movable by optimize the boundary and interior control points simultaneously \cite{liu2020simultaneous}. Zheng et al. introduced an automated technique for boundary correspondence in isogeometric analysis-suitable parameterization, grounded in optimal mass transportation theory \cite{zheng2019boundary}. This approach primarily focuses on identifying appropriate corner placements and then applies chord-length methods for boundary reparameterization, which may lead to outcomes as depicted in Figure~\ref{fig1a:chordLength}. Building on this, Zheng et al. extended the concept to volumetric cases \cite{zheng2021volumetric}, while Zhan et al. innovated the identification of corner points in the physical domain using deep neural networks \cite{zhan2023boundary}.

This paper proceeds with the fundamental assumption that the four corner points of the domain are predetermined. The key contribution of our work is in calculating the boundary parameter correspondence, a critical aspect for achieving isogeometric analysis-suitable parameterization.

\subsection{Review for Schwarz-Christoffel mapping}
\label{sec202:review_SCMap}

Conformal mapping, particularly the Schwarz-Christoffel (SC) mapping, plays a vital role due to its theoretical significance and wide-ranging practical applications \cite{lopez2022parallel}. Central to the numerical computation of SC mapping is the challenging Schwarz-Christoffel parameter problem \cite{trefethen1980numerical}. To address this complexity, a myriad of computational techniques has been developed, with numerical conformal mapping emerging as a particularly effective solution. To surmount computational challenges associated with SC mapping, Driscoll and Vavasis introduced an innovative approach utilizing Cross-Ratios and Delaunay Triangulation (CRDT) algorithms \cite{driscoll1998numerical}. This method was further extended by Delillo and Kropf to accommodate multiply connected domains \cite{delillo2011numerical}. As for the available implementations of SC mapping, the SCPACK package \cite{Trefethoen1983SCPACK}, originally developed in Fortran, represents an early solution. The SC Toolbox \cite{driscoll1996algorithm} in MATLAB and its subsequent open-source version \cite{driscoll2005algorithm} are notable advancements. The computational speed of SC mapping has been further accelerated by Banjai and Trefethen \cite{banjai2003multipole} through the use of a multipole method. Moreover, Anderson extended SC mapping to non-polygonal domains \cite{andersson2008schwarz}. For an in-depth understanding of SC mapping, the monograph by Driscoll and Trefethen \cite{driscoll2002schwarz} is a valuable reference.

\section{Problem statement and key ideas}
\label{sec3:key ideas}

In this section, the focus is on defining the specific problem addressed by our research, which establishes the notation that will be used consistently throughout the paper. Additionally, we introduce the key ideas that underpin our proposed methodology.

\subsection{Problem statement}
\label{sec301:problem statement}

In numerous industrial applications, domains characterized by a pronounced, elongated shape with extreme aspect ratios are commonly encountered. Such configurations are prevalent in channels, conduits within reactors, and waterways, as illustrated in Figure~\ref{fig1:boundary_correspondence_comparison}. In these scenarios, a reasonable boundary correspondence across opposing long boundary curves is crucial. The parameter speed along these boundaries significantly influence the quality and bijectivity of the derived parameterization. Our study specifically addresses this challenge, seeking to optimize boundary parameter matching and ensure the integrity of parameterization in such elongated domains.

Consider the real variables $\bm{x} = [x,y]^{\mathrm{T}}$ in the physical domain $\Omega$, alongside $\bm{\xi} = [\xi,\eta]^{\mathrm{T}}$, representing orthogonal real variables within the parametric domain $\hat{\Omega}$. The quartet of boundary curves defined in B-spline form, namely $\mathcal{C}^{\mathrm{West}}(\eta)$, $\mathcal{C}^{\mathrm{East}}(\eta)$, $\mathcal{C}^{\mathrm{South}}(\xi)$, and $\mathcal{C}^{\mathrm{North}}(\xi)$, with the superscript $^{*}$ representing the boundary direction, are given by
\begin{equation}
\mathcal{C}^{*}(\xi) = \sum_{i=0}^{n^{*}} \mathbf{P}_{i}^{*} R_{i,p^{*}}^{\bm{\Xi}^{*}}(\xi),
\end{equation}
where $R_{i,p^{*}}^{\bm{\Xi}^{*}}(\xi) = {\omega_i N_{i,p^{*}}^{\bm{\Xi}^{*}}(\xi)}/{\sum_{j=0}^{n^{*}} \omega_j N_{j,p^{*}}^{\bm{\Xi}^{*}}(\xi)}$ are the $i$-th degree-$p^{*}$ NURBS basis functions, and $N_{i,p^{*}}^{\bm{\Xi}^{*}}(\xi)$ denote the $i$-th degree-$p^{*}$ B-Spline basis functions defined over the knot vector $\bm{\Xi}^{*}$. Assuming, for simplicity, that the West curve $\mathcal{C}^{\mathrm{West}}(\eta)$ and the East curve $\mathcal{C}^{\mathrm{East}}(\eta)$ represent the longitudinal boundaries, our goal is to establish an analysis-suitable parameterization $\bm{x}:\ \hat{\Omega} \rightarrow \Omega$ through a coordinate transformation, expressed in B-Spline form, that adheres to the prescribed boundary conditions $\bm{x}^{-1}|_{\partial \Omega} = \partial \hat{\Omega}$. Hence, we define:
\begin{equation}
  \label{eq:3103}
  \bm{x}(\bm{\xi}) = \underbrace{ \sum_{i \in \mathcal{I}_I} \mathbf{P}_i R_i(\bm{\xi})}_{\rm unknown} + \underbrace{ \sum_{j \in \mathcal{I}_B} \mathbf{P}_j R_j(\bm{\xi})}_{\rm known},
\end{equation}
where $\mathcal{I}_I$ and $\mathcal{I}_B$ index the unknown interior and the known boundary control points, respectively,
\begin{equation}
  R_{i}(\bm{\xi}) = R_{i_1, i_2}(\xi, \eta) = \frac{\omega_{i_1, i_2} N_{i_1, p_1}^{\bm{\Xi}}(\xi) N_{i_2, p_2}^{\bm{\mathcal{H}}}(\eta)} {\sum_{i_1 = 0}^{n_1} \sum_{i_2 = 0}^{n_2} \omega_{i_1, i_2} N_{i_1, p_1}^{\bm{\Xi}}(\xi) N_{i_2, p_2}^{\bm{\mathcal{H}}}(\eta)}
\end{equation}
are bivariate tensor-product NURBS basis functions of bi-degree $(p_1,p_2)$, and $\omega_{i_1, i_2}$ are weights.

\begin{figure}[H]
    \centering
    \includegraphics[width=\linewidth]{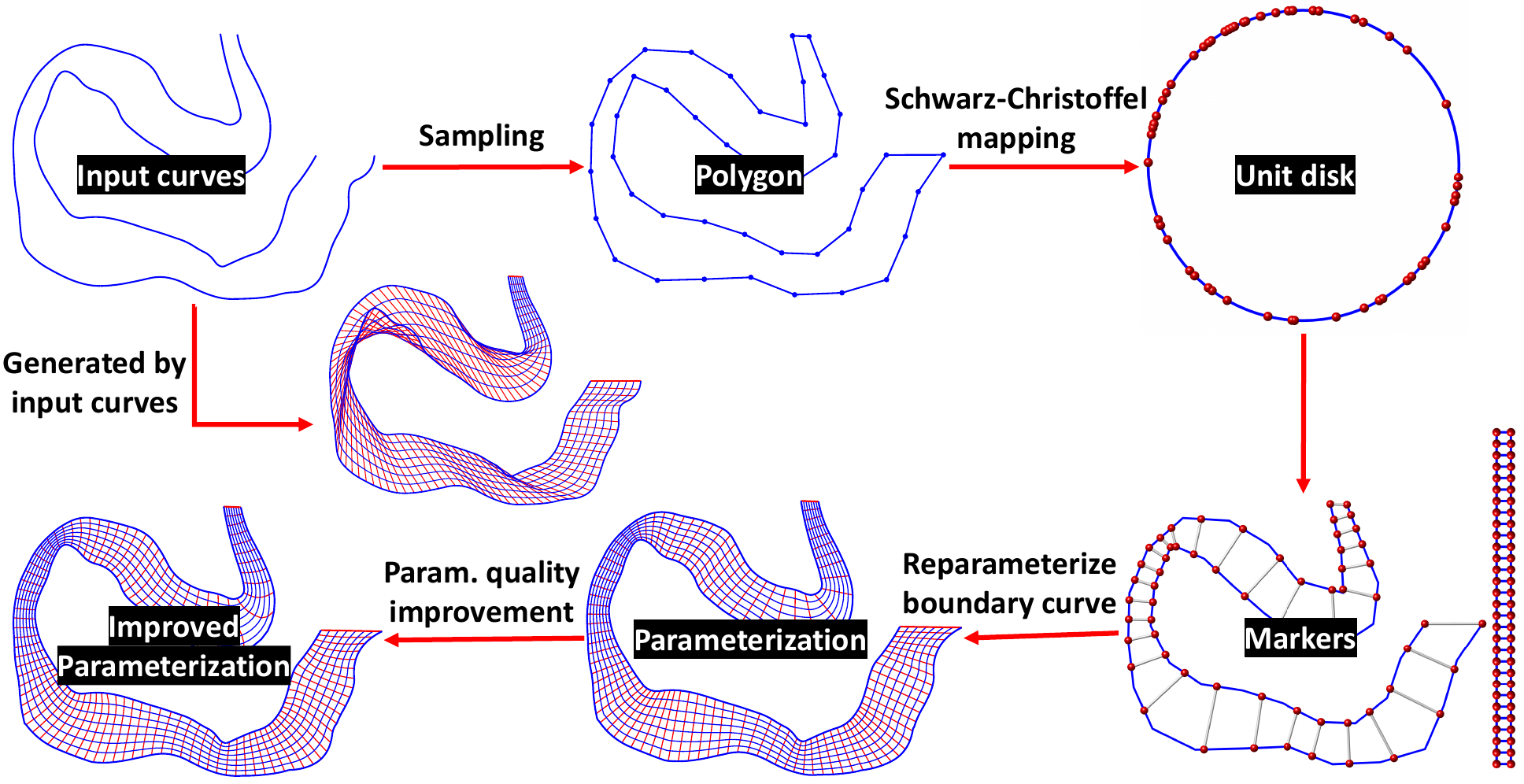}
    \caption{Workflow diagram of the proposed boundary parameter matching approach via Schwarz-Christoffel mapping.}
    \label{fig2:workflow}
\end{figure}

The differential form of the transformation is expressed as follows:
\begin{equation}
    \begin{bmatrix}
    dx \\
    dy
    \end{bmatrix} = \bm{\mathcal{J}} (\xi, \eta)
    \begin{bmatrix}
        d \xi \\
        d \eta
    \end{bmatrix}
\end{equation}
with $\bm{\mathcal{J}}(\xi, \eta)$ representing the Jacobian matrix. An orthogonal transformation is defined when
\begin{equation}
    \bm{\mathcal{J}} = \begin{bmatrix}
        h(\xi, \eta) & 0 \\
        0 & g(\xi, \eta)
    \end{bmatrix}
    \begin{bmatrix}
        \cos \theta (\xi, \eta) & - \sin \theta (\xi, \eta) \\
        \sin \theta (\xi, \eta) & \cos \theta (\xi, \eta)
    \end{bmatrix},
\end{equation}
where $h(\xi, \eta)$ and $g(\xi, \eta)$ are scalar functions that respectively define the components of the scaling transformation, and $\theta (\xi, \eta)$ represents the angle of rotation transformation.

A conformal transformation is distinguished as an orthogonal transformation that fulfills the additional condition of $g = h$. Specifically, this implies that
\begin{equation}
    \bm{\mathcal{J}} = h(\xi, \eta) 
    \begin{bmatrix}
        \cos \theta (\xi, \eta) & - \sin \theta (\xi, \eta) \\
        \sin \theta (\xi, \eta) & \cos \theta (\xi, \eta)
    \end{bmatrix}.
    \label{eq:CR_equation}
\end{equation}

Such a transformation preserves angles and is characterized by a metric that remains invariant with respect to the direction of $d \bm{\xi} = d \xi + i d\eta$. In complex context, if we define $\bm{x} = x + i y$ and $\bm{\xi} = \xi + i \eta$, then \eqref{eq:CR_equation} is essentially a statement of the well-known Cauchy–Riemann equations.

\subsection{Key ideas}
\label{sec302:key ideas}

Figure~\ref{fig2:workflow} presents the main workflow of our methodology. Central to our approach is the utilization of the Schwarz-Christoffel mapping - a conformal mapping technique. Initially, we simplify complex NURBS-represented boundary curves into a closed polygon. This step prepares the domain for the application of SC mapping. Then the polygon is numerically transposed onto a unit circle through the SC mapping process. Next, a series of markers are placed along the long boundary curves to guide the reparameterization process.

Subsequent to the mapping, we proceed with a reparameterization scheme for the boundary curves. This entails a careful recalibration of parameters, ensuring the parameterization speed of one curve is harmoniously aligned with the other. A distinguishing feature of our method is its retention to preserving both the accuracy and the continuity of the initial geometric boundaries.

Eventually, we integrate the robust and efficient parameterization technique from our previous research. This workflow yields parameterizations that are exceptionally suited to the demands of isogeometric analysis, please refer to Figure~\ref{fig2:workflow}, thereby enhancing the fidelity and utility of the analytic process.

\section{Methodology}
\label{sec4:method}

This section details the core methodology of our proposed approach. In Section~\ref{sec401:SCMapping}, we explore the foundational concepts of SC mapping. Section~\ref{sec402:SCParameterProblem} details the computational aspects of the SC parameter problem. Section~\ref{sec403:reparameterization} presents the curve reparameterization process, emphasizing its role in preserving geometric accuracy. Lastly, Section~\ref{sec404:ASParam} offers a concise overview of the improved PDE-based method.

\subsection{Schwarz–Christoffel mapping}
\label{sec401:SCMapping}



Consider $\mathcal{P}$, an open and simply connected polygon on the complex plane $\mathbb{C}$, as illustrated in the right of Figure~\ref{fig3:SC_mapping}. The Riemann mapping theorem guarantees an analytic function $f$, with a consistently non-zero derivative, mapping the open unit disk $\mathcal{D}$ onto $\mathcal{P}$, such that $f(\mathcal{D}) = \mathcal{P}$. This function $f$ is uniquely bijective over the domain $\mathcal{D}$. If we select a specific point $z_0$ in the unit disk $\mathcal{D}$ and an angle $\alpha$ in the range of $[0, 2\pi)$, then $f$ can be uniquely identified to fulfill $f(0) = z_0$ and $\arg(f'(0)) = \alpha$. Here, $z_0$ serves as the conformal center of $f$, anchoring the mapping in the complex plane.

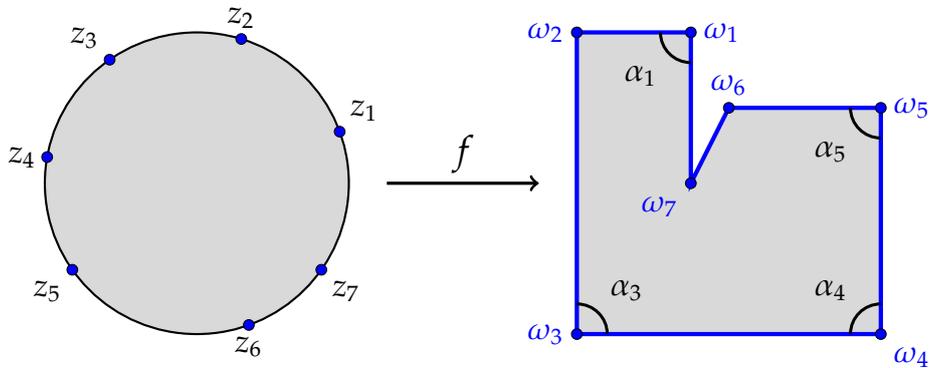
\begin{figure}[H]
    \centering
    \begin{tikzpicture}
  \draw[thick, fill=gray!30] (0,0) circle (2cm);
  
  \draw[fill=blue] (20:2cm) circle (2pt) node[anchor=south west, scale=1.2] {$z_1$};
  \draw[fill=blue] (73:2cm) circle (2pt) node[anchor=south, scale=1.2] {$z_2$};
  \draw[fill=blue] (125:2cm) circle (2pt) node[anchor=south east, scale=1.2] {$z_3$};
  \draw[fill=blue] (170:2cm) circle (2pt) node[anchor=east, scale=1.2] {$z_4$};
  \draw[fill=blue] (215:2cm) circle (2pt) node[anchor=north east, scale=1.2] {$z_5$};
  \draw[fill=blue] (290:2cm) circle (2pt) node[anchor=north, scale=1.2] {$z_6$};
  \draw[fill=blue] (325:2cm) circle (2pt) node[anchor=north west, scale=1.2] {$z_7$};
  
  \draw[->, very thick] (2.5,0) -- (4.5,0) node[midway,above] {\Large $f$};
  
  \begin{scope}[shift={(5.5,-2)}, scale=1]
    \draw[blue, ultra thick, fill=gray!30] 
    (1,4) node[right, scale=1.2] {$\omega_1$} -- 
    (-0.5,4) node[left, scale=1.2] {$\omega_2$} -- 
    (-0.5,0) node[left, scale=1.2] {$\omega_3$} -- 
    (3.5,0) node[below right, scale=1.2] {$\omega_4$} -- 
    (3.5,3) node[right, scale=1.2] {$\omega_5$} -- 
    (1.5,3) node[above, scale=1.2] {$\omega_6$} --
    (1,2) node[below left, scale=1.2] {$\omega_7$} -- cycle;

    \draw[fill=blue] (1,4) circle (2pt);
    \draw[fill=blue] (-0.5,4) circle (2pt);
    \draw[fill=blue] (-0.5,0) circle (2pt);
    \draw[fill=blue] (3.5,0) circle (2pt);
    \draw[fill=blue] (3.5,3) circle (2pt);
    \draw[fill=blue] (1.5,3) circle (2pt);
    \draw[fill=blue] (1,2) circle (2pt);

    \draw[line width=0.4mm] (1,3.6) arc (-90:-180:0.4) node[midway, below left, scale=1.2] {$\alpha_1$};
    \draw[line width=0.4mm] (-0.1,0) arc (0:90:0.4) node[midway, above right, scale=1.2] {$\alpha_3$};
    \draw[line width=0.4mm] (3.5,0.4) arc (90:180:0.4) node[midway, above left, scale=1.2] {$\alpha_4$};
    \draw[line width=0.4mm] (3.5,2.6) arc (-90:-180:0.4) node[midway, below left, scale=1.2] {$\alpha_5$};
  \end{scope}
\end{tikzpicture}
    \caption{Notational conventions for the Schwarz–Christoffel mapping.}
    \label{fig3:SC_mapping}
\end{figure}

Denote by $\omega_1 < \omega_2 < \cdots < \omega_n$ the vertices of the polygon $\mathcal{P}$ in counterclockwise order. For ease of indexing purposes, we define $\omega_{n+1} = \omega_1$ and $\omega_0 = \omega_n$. The interior angles at these vertices are represented by $\alpha_1, \alpha_2, \ldots, \alpha_n$. We define $\beta_j = \alpha_j / \pi - 1$ for each vertex $j$, where $\beta_j$ falls within the interval $(-1, 1]$. Furthermore, the sum of the interior angles must equate to a total turn of $2\pi$, satisfying $\sum_{k=1}^n \alpha_k = n - 2$. Given these constraints, a conformal mapping $f$ from the unit disk $\mathcal{D}$ to the polygon $\mathcal{P}$, as shown in Figure~\ref{fig3:SC_mapping}, can be described as \cite{driscoll2002schwarz}:
\begin{equation}
    f(z) = f(z_0) + C \int_{z_0}^{z} \prod_{j=1}^{n} \left( 1 - \frac{\zeta}{z_j} \right)^{\beta_j} d \zeta,
    \label{eq4101:SC formula}
\end{equation}
where $f(z_0)$ and $C$ are complex constants with $C \neq 0$. Here, $z_1, z_2, \ldots, z_n$ are prevertices located on the boundary of $\mathcal{D}$ in counterclockwise order, satisfying the condition $\omega_k = f(z_k)$ for each $k = 1, 2, \ldots, n$. Equation~\eqref{eq4101:SC formula} is commonly referred to as the \textit{Schwarz-Christoffel formula}.

\subsection{Schwarz-Christoffel parameter problem}
\label{sec402:SCParameterProblem}

A crucial step of the transformation \eqref{eq4101:SC formula} is the determination of the prevertices $z_j$, for $j =1, 2, \ldots, n$. These prevertices typically lack a closed-form solution, except in special cases. The transformation is characterized by $n+4$ real parameters, including the affine constants $f(z_0)$ and $C$, as well as the angles $\theta_1, \theta_2, \ldots, \theta_n$. Each $\theta_i$ denotes the argument of the complex number $\omega_i$. Due to the three degrees of freedom inherent in M\"{o}bius transformations, it is possible to arbitrarily select three prevertices, including the predetermined $z_n$. This choice results in $n-3$ remaining prevertices, which are uniquely defined and necessitate solving a system of nonlinear equations. This challenge forms the essence of the \textit{Schwarz-Christoffel parameter problem}.

Practical implementation of SC mapping often involves solving complex nonlinear equations, derived from the geometric constraints of polygons. This process, fundamental to software like SCPACK \cite{Trefethoen1983SCPACK} and the SC Toolbox for Matlab \cite{driscoll1996algorithm}, aims to match the computed properties of the polygon $\mathcal{P}$ - such as side lengths and orientation - with those of the desired polygon, thereby establishing $n-3$ conditions \cite{trefethen1980numerical}. However, two significant challenges limit the efficiency and applicability of these software packages.

Firstly, the nonlinear systems often lack a simple, solvable structure, leading to potential pitfalls such as local minima, which can significantly hinder, or even entirely prevent, the convergence of solvers. Secondly, a more critical challenge is the ``crowding'' phenomenon, a frequent occurrence in domains characterized by elongated and narrow channels \cite{howell1990modified}. This effect causes the positions of the prevertices $z_j$ to be disproportionately skewed, with their displacement being exponentially proportional to the aspect ratio of these constricted regions. Such intricacies in the conventional SC mapping approaches underscore the need for a more robust and reliable alternative. In response to these challenges, our approach adopts the Cross-Ratios of the Delaunay Triangulation (CRDT) algorithm, a method within the numerical conformal mapping paradigm, as proposed in \cite{driscoll1998numerical}. This method addresses the aforementioned issues more effectively by leveraging the stability and precision of the CRDT technique.

The CRDT hinges on the principle that the cross-ratio is invariant under conformal mapping. Consider four distinct points $a$, $b$, $c$, and $d$ in the complex plane, arranged to form a quadrilateral $abcd$ with vertices ordered counterclockwise. Additionally, let $ac$ be an interior diagonal of this quadrilateral. The cross-ratio of these points is mathematically defined as follows:
\begin{equation}
\rho(a, b, c, d) = \frac{(d - a) (b - c)}{(c - d) (a - b)}.
\end{equation}

The main computation scheme is as follows:

\begin{itemize}
    \item Step 1: Splitting the long edges of polygon $\mathcal{P}$ prevents the formation of elongated and narrow quadrilaterals whose edges align with those of the polygon. This step is crucial as it ensures the prevertices of such quadrilaterals are not densely crowded on the unit circle. After this splitting, let $n$ represent the total number of vertices.
    \item Step 2: Compute the Delaunay triangulation of the polygon $\mathcal{P}$. From this triangulation, identify the $n-3$ diagonals, denoted as $d_1, d_2, \ldots, d_{n-3}$, and their corresponding quadrilaterals $Q_1, Q_2, \ldots, Q_{n-3}$, with $Q_i = Q(d_i)$ for each $i$. For each quadrilateral $Q_i$, the vertices are represented as $\omega_{k(i,1)}, \omega_{k(i,2)}, \omega_{k(i,3)},$ and $\omega_{k(i,4)}$, where $i$ ranges from $1$ to $n-3$. Here, each four-tuple $\left( k(i,1), k(i,2), k(i,3), k(i,4) \right)$ consists of distinct indices drawn from the set ${1, 2, \ldots, n}$. The next step is to calculate $c_i$ for each quadrilateral, defined as:
    \begin{equation}
    c_i = \log \left( \left \vert \rho(\omega_{k(i,1)}, \omega_{k(i,2)}, \omega_{k(i,3)}, \omega_{k(i,4)}) \right \vert \right),
    \end{equation}
    for $i=1,2,\ldots,n-3$, where $\vert \cdot \vert$ indicates the magnitude of a complex number.
    \item Step 3: Solve the nonlinear system $\bm{\mathcal{F}} = 0$, where the $i$-th nonlinear equation is 
    \begin{equation}
        \bm{\mathcal{F}}_i = \log \left( \left \vert \rho (\zeta_{k(i,1)}, \zeta_{k(i,2)}, \zeta_{k(i,3)}, \zeta_{k(i,4)}) \right \vert \right) - c_i,
        \label{eq4203:nonlinearSystem}
    \end{equation}
    $\zeta_1, \zeta_2, \ldots, \zeta_n$ constitute the vertices of the polygonal image derived from the SC mapping formula \eqref{eq4101:SC formula}. This is based on the invariance of the cross-ratio among the vertices $\zeta$ under a similarity transformation. For integration along a straight-line path originating from the origin, the compound Gauss–Jacobi quadrature rule, as detailed by \cite{trefethen1980numerical}, is utilized. The central computational challenge in CRDT lies in effectively solving this nonlinear system $\bm{\mathcal{F}} = 0$ in \eqref{eq4203:nonlinearSystem}. In practice, we found that the simple Picard iteration converges too slowly. Therefore, we opt for nonlinear solvers that rely solely on function evaluations. Specifically, we employ the Gauss-Newton solver, enhanced with a Broyden update for the derivative $\bm{\mathcal{F}}'$, to achieve more efficient convergence.
\end{itemize}

Upon computing the SC mapping as described in Eq.~\eqref{eq4101:SC formula}, we then evaluate equation \eqref{eq4101:SC formula} to identify the marker sets $P_i^{\rm West}$ and $P_i^{\rm East}$, located on the boundary curves $\mathcal{C}^{\rm West}$ and $\mathcal{C}^{\rm East}$, respectively, as depicted in Figure~\ref{fig4a:curveReparam}. The next step in our methodology is to reparameterize one boundary curve while keeping the parameter speed of the other boundary fixed.

\subsection{Geometry-preserving curve reparameterization technique}
\label{sec403:reparameterization}

Without loss of generality, we choose to fix the West boundary curve $\mathcal{C}^{\rm West}$ and apply reparameterization to the East boundary curve $\mathcal{C}^{\rm East}$. The theoretical foundation of the reparameterization process is the invariant property of B-Spline basis functions under scaling and translation transformations, as elucidated in the following lemma:

\begin{lemma}
    Let $N^{\bm{\Xi}}_{i, p}(\xi)$ be the $i$-th degree-$p$ B-Spline basis function defined over the knot vector $\bm{\Xi}$. Consider a scaled and translated knot vector $\hat{\bm{\Xi}} = s \bm{\Xi} + t$ with $s > 0$. Then, $N^{\hat{\bm{\Xi}}}_{i, p}(s \xi + t) = N^{\bm{\Xi}}_{i, p}(\xi)$ holds.
    \label{lemma1:invariant}
\end{lemma}
\begin{proof}
    We prove this lemma using the principle of mathematical induction on the degree $p$ of the B-Spline basis functions.
    
    For $p=0$, $N^{\bm{\Xi}}_{i, 0}(\xi)$ and $N^{\hat{\bm{\Xi}}}_{i, 0}(s \xi + t)$ are piecewise constant functions defined as:
    \begin{equation*}
        N^{\bm{\Xi}}_{i, 0}(\xi) = 
        \begin{cases}
        1 & \text{if } \xi_{i} \leq \xi < \xi_{i+1}, \\
        0 & \text{otherwise},
        \end{cases}
    \end{equation*}
    and
    \begin{equation*}
        N^{\hat{\bm{\Xi}}}_{i, 0}(s \xi + t) = 
        \begin{cases}
        1 & \text{if } s \xi_{i} + t \leq s \xi + t < s \xi_{i+1} + t, \\
        0 & \text{otherwise}.
        \end{cases}
    \end{equation*}
    Since $s > 0$, the conditions for non-zero values are equivalent, implying $N^{\hat{\bm{\Xi}}}_{i, 0}(s \xi + t) = N^{\bm{\Xi}}_{i, 0}(\xi)$.
    
    Assume the lemma holds for all degrees less than $p \geq 1$. Then, for degree $p$, the recurrence relation for B-Spline basis functions gives:
    \begin{equation}
    \begin{aligned}
        N^{\hat{ \bm{\Xi} }}_{i, p}(s \xi + t) &= \frac{(s \xi + t) - (s \xi_{i} + t)}{(s \xi_{i+p} + t) - (s \xi_{i} + t)} N^{\hat{ \bm{\Xi} }}_{i, p-1}(s \xi + t) + \frac{(s \xi_{i+p+1} + t) - (s \xi + t)}{(s \xi_{i+p+1} + t) - (s \xi_{i+1} + t)} N^{\hat{ \bm{\Xi} }}_{i+1, p-1}(s \xi + t) \\
        &= \frac{s (\xi - \xi_{i})}{s (\xi_{i+p} - \xi_{i})} N^{\bm{\Xi}}_{i, p-1}(\xi) + \frac{s (\xi_{i+p+1} - \xi)}{s (\xi_{i+p+1} - \xi_{i+1})} N^{ \bm{\Xi}}_{i+1, p-1}(\xi) \\
        &= N^{\bm{\Xi}}_{i, p}(\xi).
    \end{aligned}
    \end{equation}
    This completes the proof.
\end{proof}

\begin{corollary}
    Let $R^{\bm{\Xi}, \bm{\mathcal{H}}}_{i,j,p,q}(\xi, \eta)$ be the NURBS basis function defined over the knot vectors $\bm{\Xi}$ and $\bm{\mathcal{H}}$ with weights $\omega_{i,j}$, and degrees $p$ and $q$. Consider scaled and translated knot vectors $\hat{\bm{\Xi}} = s \bm{\Xi} + t$ and $\hat{\bm{\mathcal{H}}} = s' \bm{\bm{\mathcal{H}}} + t'$ with $s> 0$ and $s' > 0$. Then, $R^{\hat{\bm{\Xi}}, \hat{\bm{\mathcal{H}}}}_{i,j,p,q}(s \xi + t, s' \eta + t') = R^{\bm{\Xi}, \bm{\mathcal{H}}}_{i,j,p,q}(\xi, \eta)$.
    \label{corollary2:invariant}
\end{corollary}
\begin{proof}
    Given the invariant property of B-Spline basis functions as stated in Lemma~\ref{lemma1:invariant}, the proof of this corollary is trivial.
\end{proof}

Building upon the invariant property of B-Spline basis functions established in Lemma~\ref{lemma1:invariant}, we can extend this concept to NURBS curves. The invariance of B-Spline basis functions under scaling and translation transformations of the knot vector implies that NURBS curves, which are defined using these basis functions, will similarly maintain their geometric form under such transformations. This observation leads us to the following proposition:

\begin{proposition}
Let $\mathcal{C}(\xi)$ be a NURBS curve defined over the parameter $\xi$ within the domain of the knot vector $\bm{\Xi}$. If $\hat{\bm{\Xi}} = s \bm{\Xi} + t$ represents the scaled and translated knot vector with $s > 0$, then the NURBS curve $\hat{\mathcal{C}}(s \xi + t)$, defined over the transformed parameter domain, is geometrically identical to $\mathcal{C}(\xi)$. Formally, $\hat{\mathcal{C}}(s\xi+t) = \mathcal{C}(\xi)$, ensuring the geometric shape of the curve remains invariant under scaling and translation of its parameter domain.
\label{prop1:NURBScurveInvariance}
\end{proposition}

Based on the fundamental principle in Proposition \ref{prop1:NURBScurveInvariance} that the geometric integrity of NURBS curves is preserved under scaling and translation transformations of their parameter domain, we develop a novel reparameterization method specifically tailored for addressing boundary parameter correspondence challenges.

\begin{figure}[H]
    \centering
    \subfigure[Schematic diagram of curve reparameterization]{\includegraphics[width=0.4\linewidth]{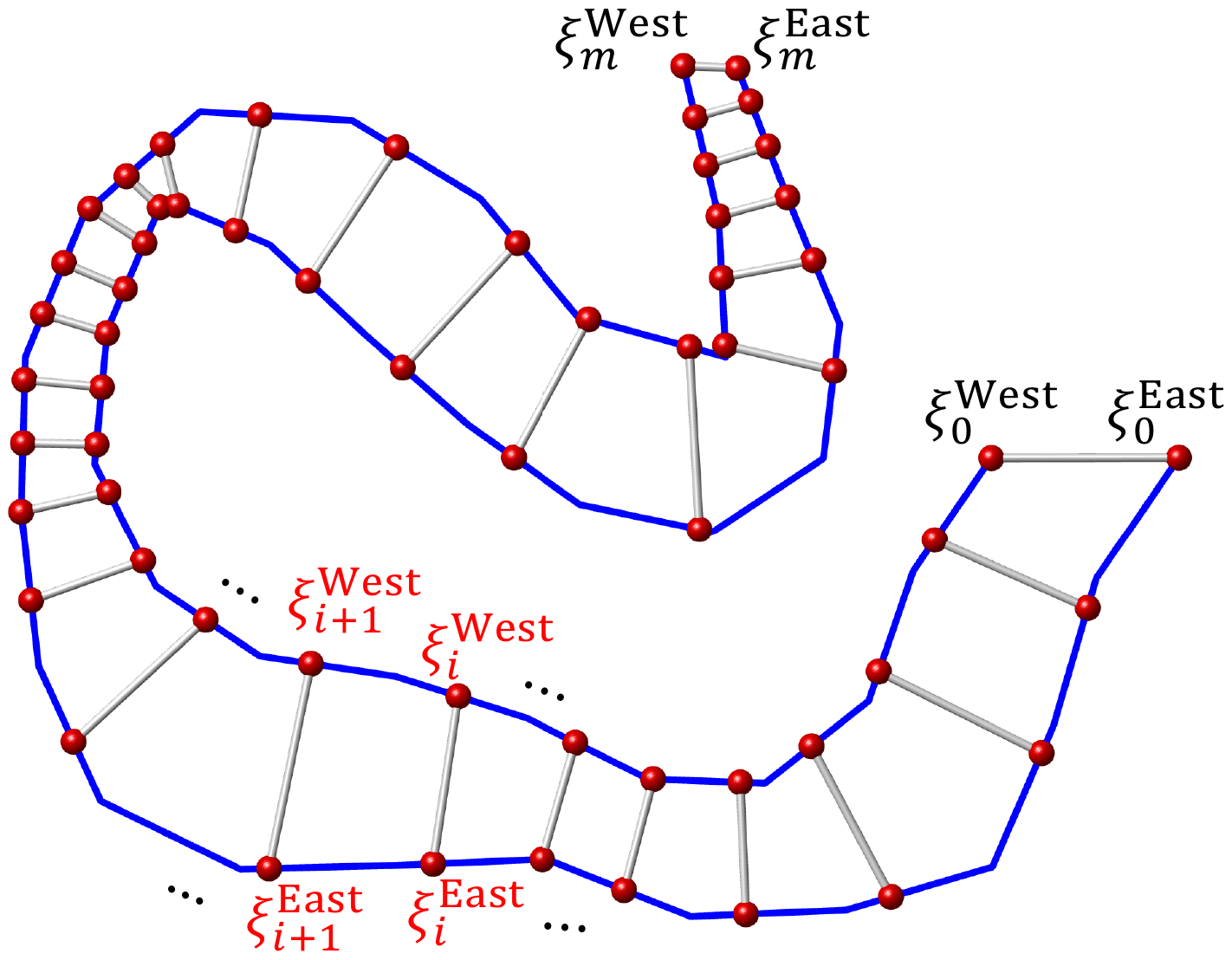}
    \label{fig4a:curveReparam}}
    \quad 
    \subfigure[Parameterization after curve reparameterization]{
    \includegraphics[width=0.4\linewidth]{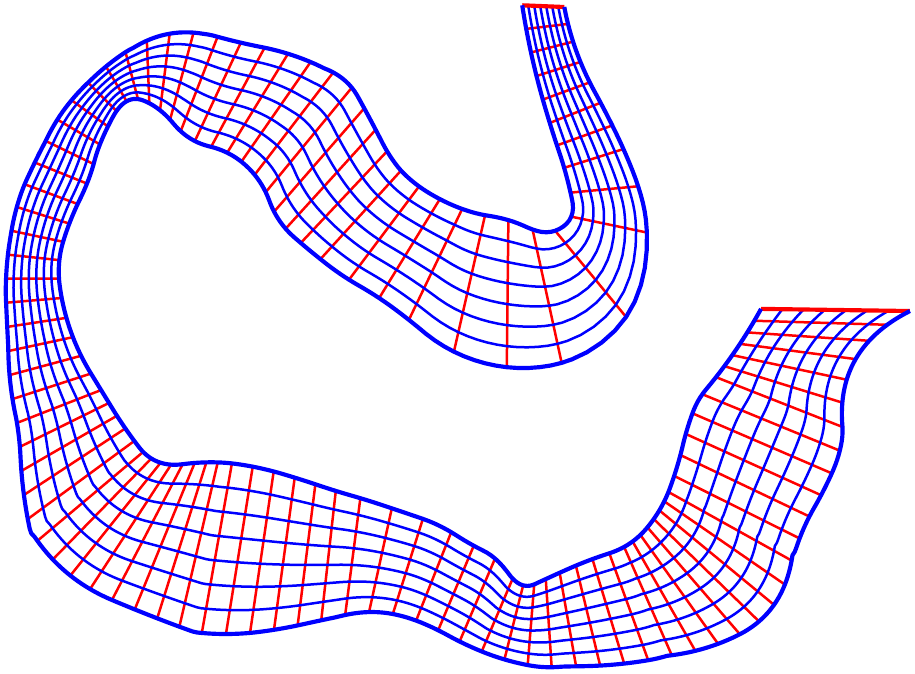}
    \label{fig4b:resultParam}}
    \caption{Schematic diagram for curve reparameterization and the resulting parameterization after curve parameterization.}
    \label{fig4:curveReparam}
\end{figure}

For each pair of the markers, as shown in Figure~\ref{fig4a:curveReparam}, we determine the parameters $\xi_i^{\rm East}$ and $\xi_i^{\rm West}$, corresponding to the closest points on the East and West boundary curves, respectively. This is achieved by resorting the nonlinear equation:
\begin{equation}
    \left( \mathcal{C}^{*}(\xi) - \mathbf{P}^{*}_i \right) \cdot \mathcal{C}^{', *}(\xi) = 0,
    \label{eq4301:closestPt}
\end{equation}
where $\mathcal{C}^{', *}(\xi) = \partial \mathcal{C}^{*}(\xi) / \partial \xi$ represents the first derivative of the curve $\mathcal{C}^{*}(\xi)$, $\mathbf{P}^{*}_i$ is the $i$-th marker on the boundary curve $\mathcal{C}^{*}(\xi)$. In our implementation, the standard Newton's method is employed to solve equation \eqref{eq4301:closestPt} efficiently. 

Upon identifying these parameters, we then focus on the East boundary curve $\mathcal{C}^{\rm East}$. We split this curve at the parameter $\xi_{i+1}^{\rm East}$. This is achieved by inserting the value of $\xi_{i+1}^{\rm East}$ until its multiplicity is equal to the degree of the curve. Consequently, we obtain a curve segment defined over the parameter interval $[\xi_{i}^{\rm East}, \xi_{i+1}^{\rm East}]$. The next step is to align this segment with the West boundary curve $\mathcal{C}^{\rm West}$. To accomplish this, we apply an affine transformation to the curve segment. This transformation adjusts the parameter interval of the East curve segment to match the parameter interval $[\xi_{i}^{\rm West}, \xi_{i+1}^{\rm West}]$. The result is a harmonized alignment between the East and West boundary curves, crucial for maintaining geometric consistency in the reparameterization process.

This procedure enables the reparameterization of the curve while maintaining the geometric integrity. By applying linear interpolation-based parameterization method along the $\eta$-direction, we achieve a refined parameterization. This improvement is demonstrated in Figure~\ref{fig4b:resultParam}, highlighting the significant parameterization quality improvement achieved through our reparameterization procedure.

\subsection{Analysis-suitable paramterization}
\label{sec404:ASParam}

As demonstrated in Figure~\ref{fig5a:linearInterp}, our method for reparameterizing boundary curves, when compared with the original boundary curves shown in Figure~\ref{fig1a:chordLength}, proves to be highly effective. Even a straightforward linear interpolation between corresponding boundary control points results in a markedly improved parameterization. In this section, we explore the possibility of further enhancing the parameterization quality by integrating our previously developed PDE-based method \cite{ji2023improved}.

It is important to highlight that, in this case, we are constrained by the absence of additional degrees of freedom to refine the parameterization quality. To address this, we implement $k$-refinements along the $\eta$-direction, thereby introducing the necessary degrees of freedom. Henceforth in this paper, we will default to elevating the degree along $\eta$-direction to $2$ and inserting $2$ extra knots to the knot vector $\bm{\Xi}$, except where explicitly stated otherwise. This $k$-refinement procedure results in the formation of the knot vector $\bm{\mathcal{H}} = \{ 0, 0, 0, 1/3, 2/3, 1, 1, 1 \}$.

The core step of the PDE-based parameterization method lies in solving a nonlinear elliptic PDE system:
\begin{equation}
\left \{ 
\begin{array}{cc}
\nabla \cdot [\mathbb{A} \nabla \xi(x)] &= 0\\
\nabla \cdot [\mathbb{A} \nabla \eta(x)] &= 0 
\end{array}
\quad \text{s.t.}\ \bm{x}^{-1}|_{\partial \Omega} = \partial \hat{\Omega}.
\right.
\label{eq4401:quasiHaramonicEqs}
\end{equation}
where the tensor field $\mathbb{A}$ is set to $\texttt{DIAG}(1/\left \vert \bm{\mathcal{J}} \right \vert, 1/\left \vert \bm{\mathcal{J}} \right \vert)$ in \cite{ji2023improved}.


\begin{figure}[H]
\centering
\subfigure[Linear interpolation-based parameterization]{\includegraphics[width=0.5\linewidth]{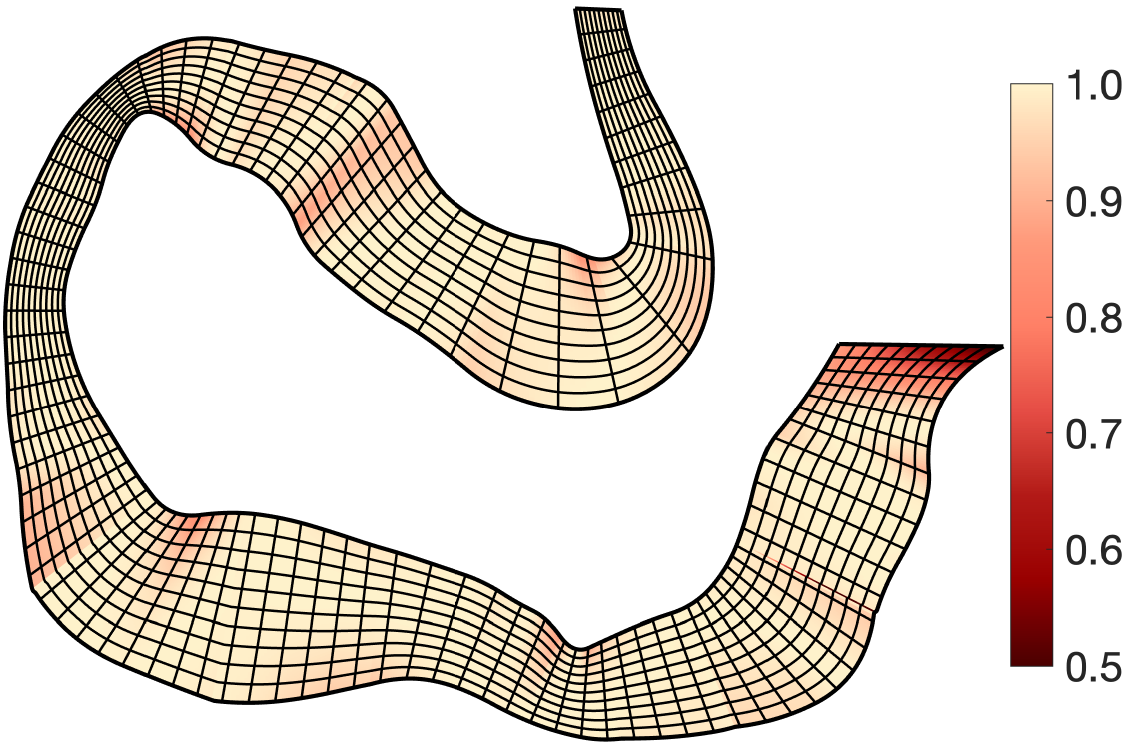}
\label{fig5a:linearInterp}}
\quad
\subfigure[Improved parameterization via PDE-Based approach]{
\includegraphics[width=0.45\linewidth]{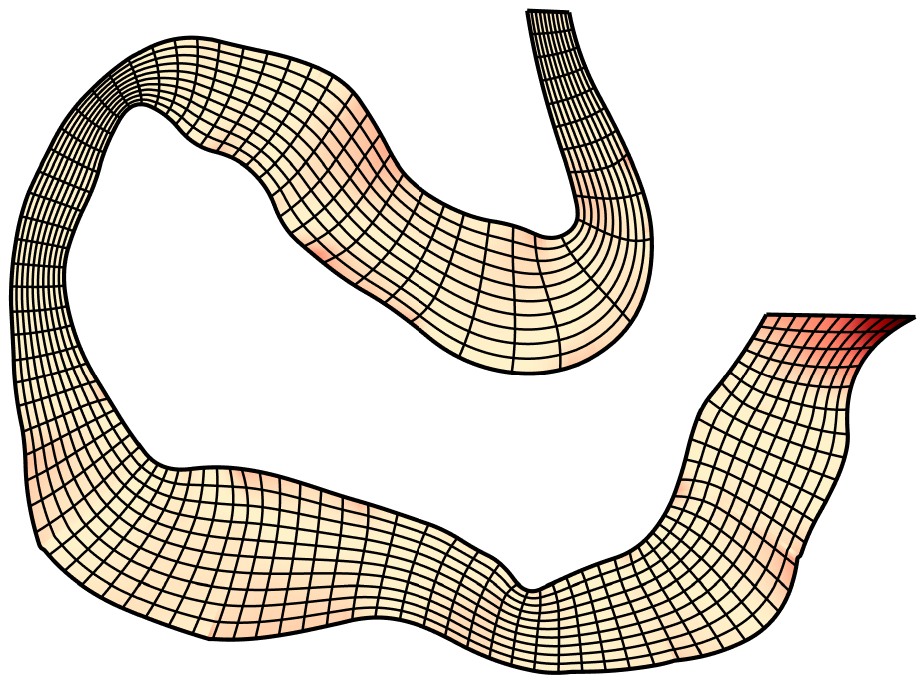}
\label{fig5b:PDEApproach}}
\caption{Comparative visualization of scaled Jacobian $\vert \bm{\mathcal{J}} \vert_s$: (a) Parameterization via linear interpolation; (b) Improved parameterization employing PDE-Based approach.}
\label{fig5:paramResult}
\end{figure}

Denote by $\mathbb{S}_{p,q}^{\bm{\Xi}, \bm{\mathcal{H}}}$ the spline space spanned by the bivariate B-spline bases of degree $p$ and $q$ with knot vectors $\bm{\Xi}$ and $\bm{\mathcal{H}}$, respectively, and $(\mathbb{S}_{p,q}^{\bm{\Xi}, \bm{\mathcal{H}}})^0 = \{N_i \in \mathbb{S}_{p,q}^{\bm{\Xi}, \bm{\mathcal{H}}}: N_i \vert _{\partial \hat{\Omega}} = 0 \}$ be those vanishing on $\partial \Omega$. The variational principle deduces the nonlinear PDE system \eqref{eq4401:quasiHaramonicEqs} to its equivalent discretization scheme:

\begin{equation}
\forall N_i \in (\mathbb{S}_{p,q}^{\bm{\Xi}, \bm{\mathcal{H}}})^0: 
\left \{
\begin{array}{cc}
     \int_{\Omega}\ \nabla \mathbf{N} \cdot \mathbb{A} \nabla \xi (x)\ \rm{d} \Omega = \mathbf{0}, &\\
     \int_{\Omega}\ \nabla \mathbf{N} \cdot \mathbb{A} \nabla \eta (x)\ \rm{d} \Omega = \mathbf{0}, &
\end{array}
\quad
\text{s.t.}\ \bm{x}|_{\partial \hat{\Omega}} = \partial \Omega,
\right.
\label{eq4402:nonlinearSystem}
\end{equation}
where $\mathbf{N}$ denotes the collection of B-spline basis functions in the spline space $(\mathbb{S}_{p,q}^{\bm{\xi}, \bm{\mathcal{H}}})^0$. The enhanced parameterization, depicted in Figure~\ref{fig5b:PDEApproach}, demonstrates a significant improvement in parameterization quality. This advancement underscores the effectiveness of the proposed approach in refining the parameterization process. Our implementation is available now in the open-source \texttt{G+Smo} library \cite{juttler2014geometry}, offering a robust solution framework for the analysis-suitable parameterization problem in isogeometric analysis.

\section{Numerical experiments and comparisons}
\label{sec5:experiments}

In this section, we present a numerical investigation to demonstrate the effectiveness of our proposed method for boundary parameter matching problem.

\subsection{Implementation details and parameterization quality metrics}
\label{sec501:implementation}

The method described in this paper is implemented using C++ programming language. The computational experiments were conducted on a MacBook Pro 14-inch 2021 equipped with an Apple M1 Pro CPU and 16 GB of RAM. We utilize \texttt{G+Smo} (Geometry + Simulation Modules), an open-source C++ library, as the foundation of our implementation, leveraging its extensive functionalities in geometric computing and simulation \cite{juttler2014geometry, mantzaflaris2019overview}. For handling matrix and vector operations, as well as solving the linear systems integral to our method, Eigen library \cite{eigenweb} is integrated.

Two quality metrics are employed to assess the parameterization: the scaled Jacobian, denoted as $\vert \bm{\mathcal{J}} \vert_s$, which evaluates orthogonality, and the uniformity metric, denoted as $m_{\rm unif.}$, which measures uniformity.

\begin{itemize}
    \item The scaled Jacobian is defined as follows:
    \begin{equation}
    \vert \bm{\mathcal{J}} \vert_s = \frac{\vert \bm{\mathcal{J}} \vert}{\vert \bm{x}_{\xi} \vert \vert \bm{x}_{\eta} \vert}.
    \end{equation}
    This metric takes values ranging from $-1.0$ to $1.0$. A negative value of $\vert \bm{\mathcal{J}} \vert_s$ indicates a fold in the parameterization $\bm{x}$, signifying overlaps in the mapping. The ideal value of $\vert \bm{\mathcal{J}} \vert_s$, which is $1.0$, is achieved when the orthogonality in the parameterization is maintained optimally.
    \item The uniformity metric is computed as:
    \begin{equation}
    m_{\rm unif.} = \left \vert \frac{\vert \mathcal{J} \vert}{R_{\rm area}} - 1 \right \vert ,
    \end{equation}
    where $R_{\rm area} = {\text{Area}(\Omega)}/{\text{Area}(\hat{\Omega})}$ represents the area ratio between the physical domain $\Omega$ and the parametric domain $\hat{\Omega}$. This metric $m_{\rm unif.}$ attains its optimal value of $0.0$ when the parameterization uniformly conserves area.
\end{itemize}

For a comprehensive evaluation, these metrics were calculated over a dense grid of $1001 \times 1001$ sample points, including the domain boundaries. In our analysis, we exclude the maximum values of $\vert \bm{\mathcal{J}} \vert_s$ and the minimum values of $m_{\rm unif.}$ from the statistics, as these are typically achievable in all the examples.

\subsection{Role of boundary correspondence}
\label{sec502:comparison}

This section is devoted to conducting a comparative quality analysis of the parameterizations derived from various methods.

Initially, as depicted in Figure~\ref{fig6a:Scheldt_inputSJ}, the boundary curves are parameterized using an approximate chord-length method tailored to their individual representations. However, this method faces challenges in creating analysis-suitable parameterizations, mainly due to a mismatch in parameter speeds that degrades the parameterization quality. This degradation becomes particularly noticeable in the resulting poor parameterization.

\begin{figure}[H]
    \centering
    \subfigure[Scaled Jacobian $\vert \bm{\mathcal{J}} \vert_s$ (left) and uniformity metric $m_{\rm unif.}$ (right) for initial chord-length parameterization]{
    \includegraphics[width=0.45\linewidth]{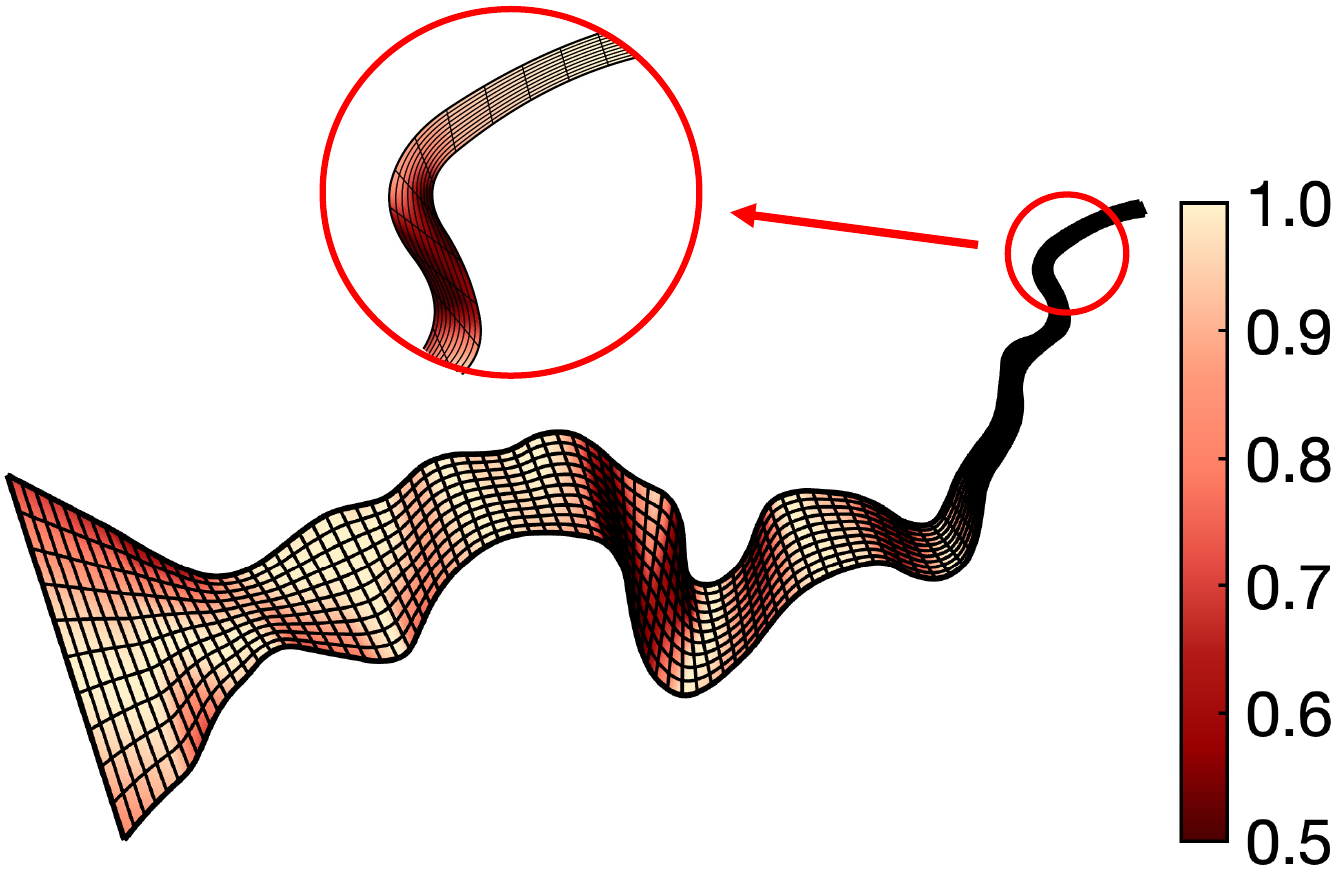}
    \quad 
    \includegraphics[width=0.45\linewidth]{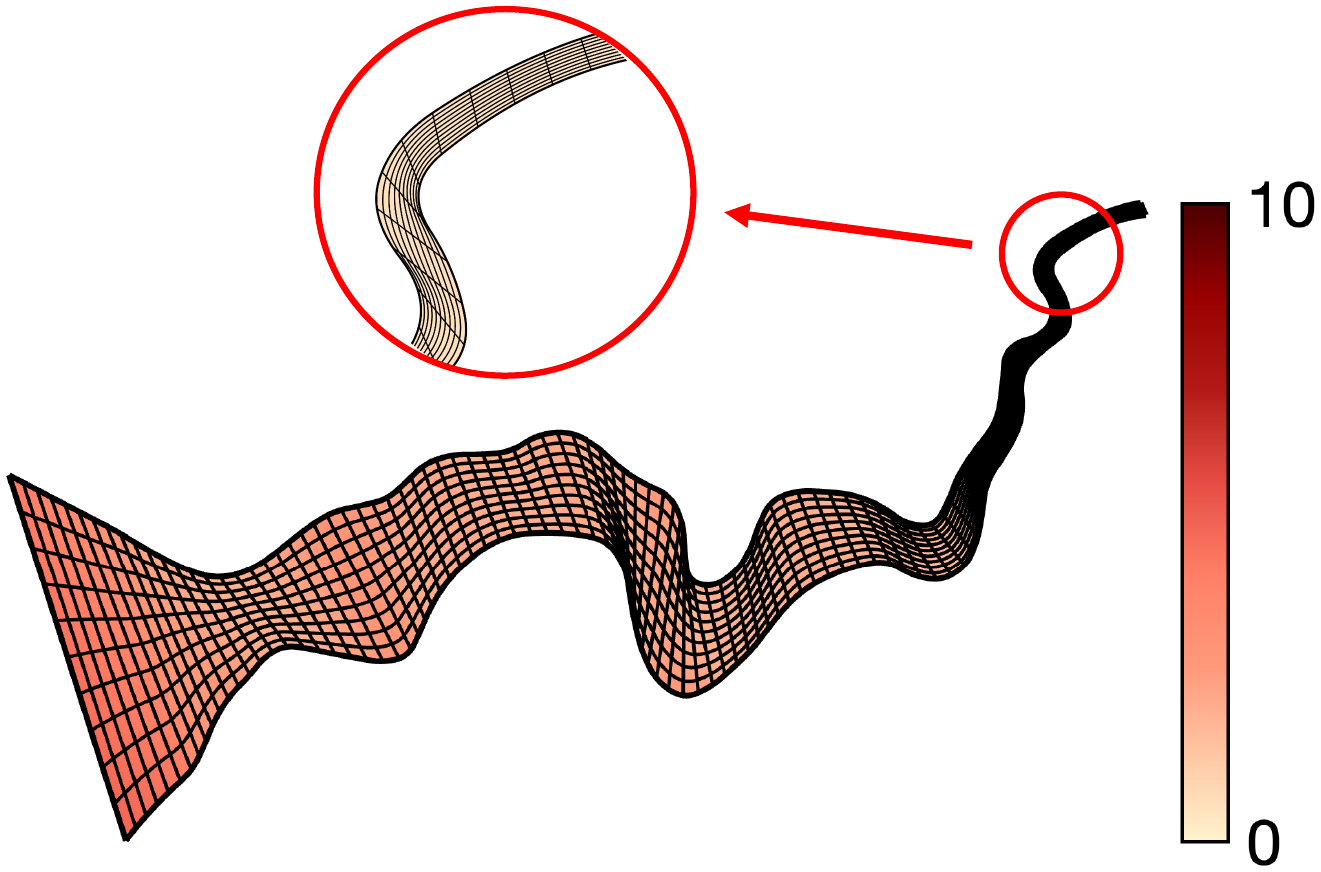}
    \label{fig6a:Scheldt_inputSJ}}
    \subfigure[Scaled Jacobian $\vert \bm{\mathcal{J}} \vert_s$ (left) and uniformity metric $m_{\rm unif.}$ (right) for linear interpolation-based parameterization]{
    \includegraphics[width=0.45\linewidth]{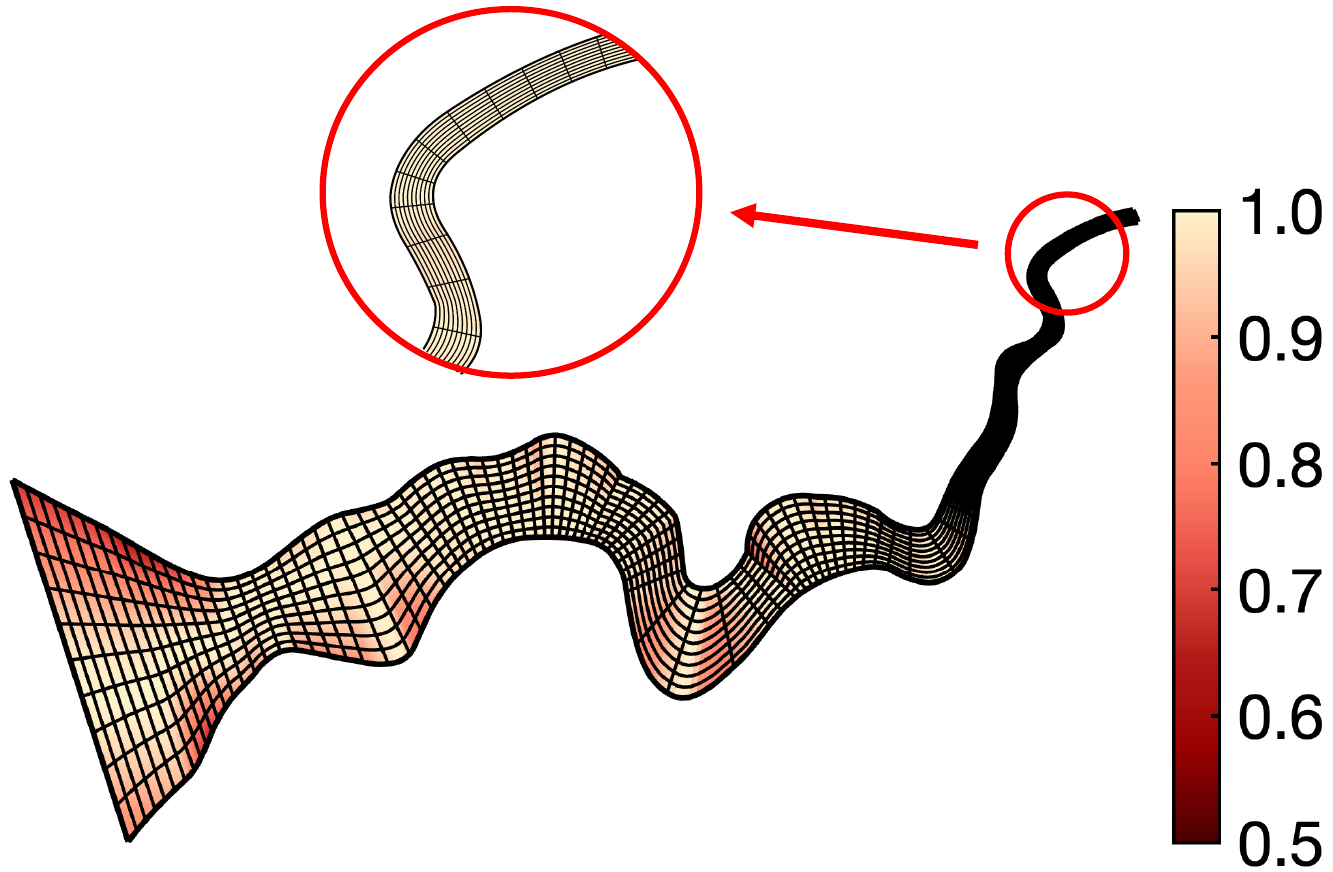}
    \quad 
    \includegraphics[width=0.45\linewidth]{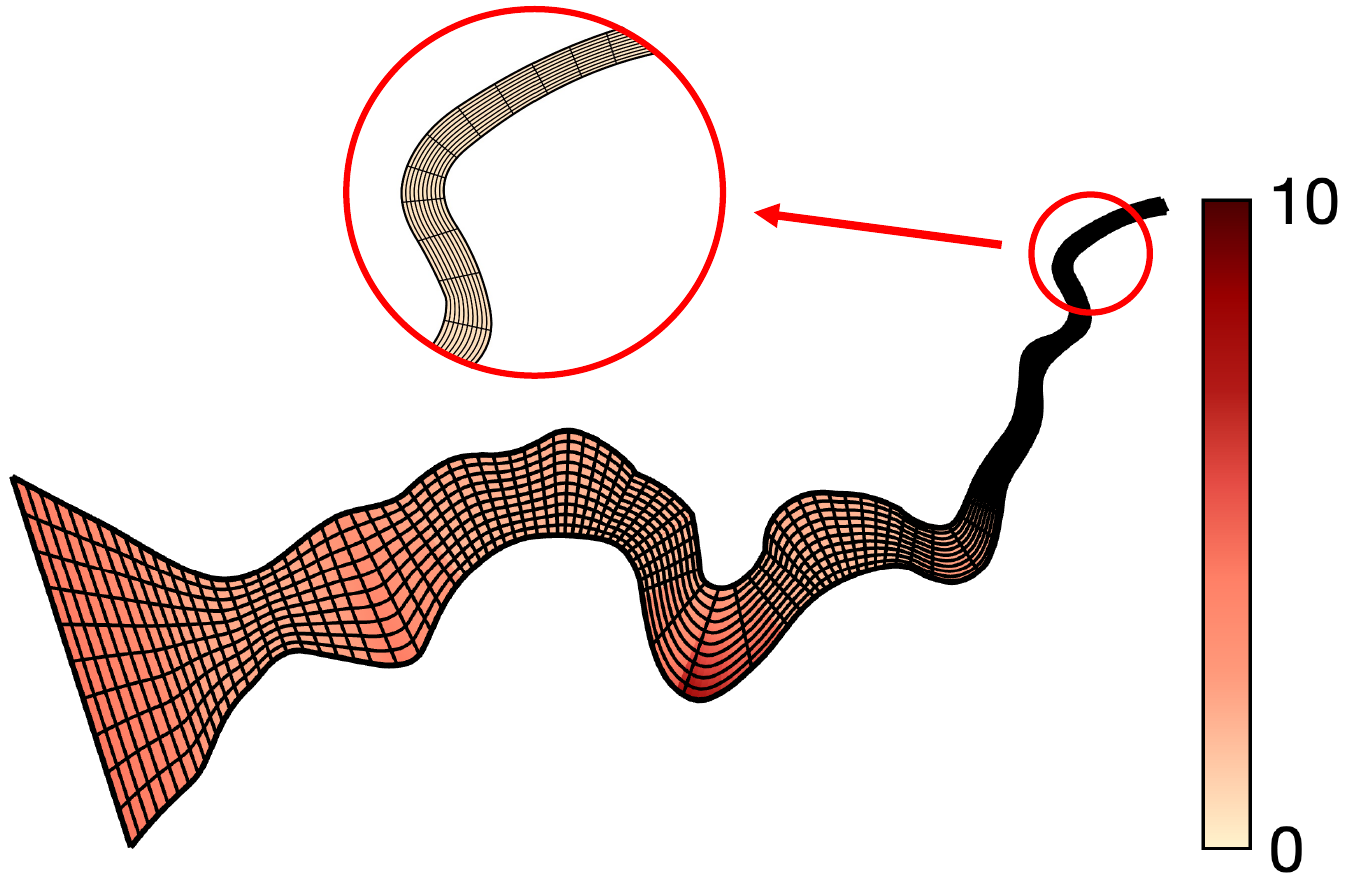}
    \label{fig6b:Scheldt_linearSJ}}
    \subfigure[Scaled Jacobian $\vert \bm{\mathcal{J}} \vert_s$ (left) and uniformity metric $m_{\rm unif.}$ (right) for PDE-based parameterization approach]{
    \includegraphics[width=0.45\linewidth]{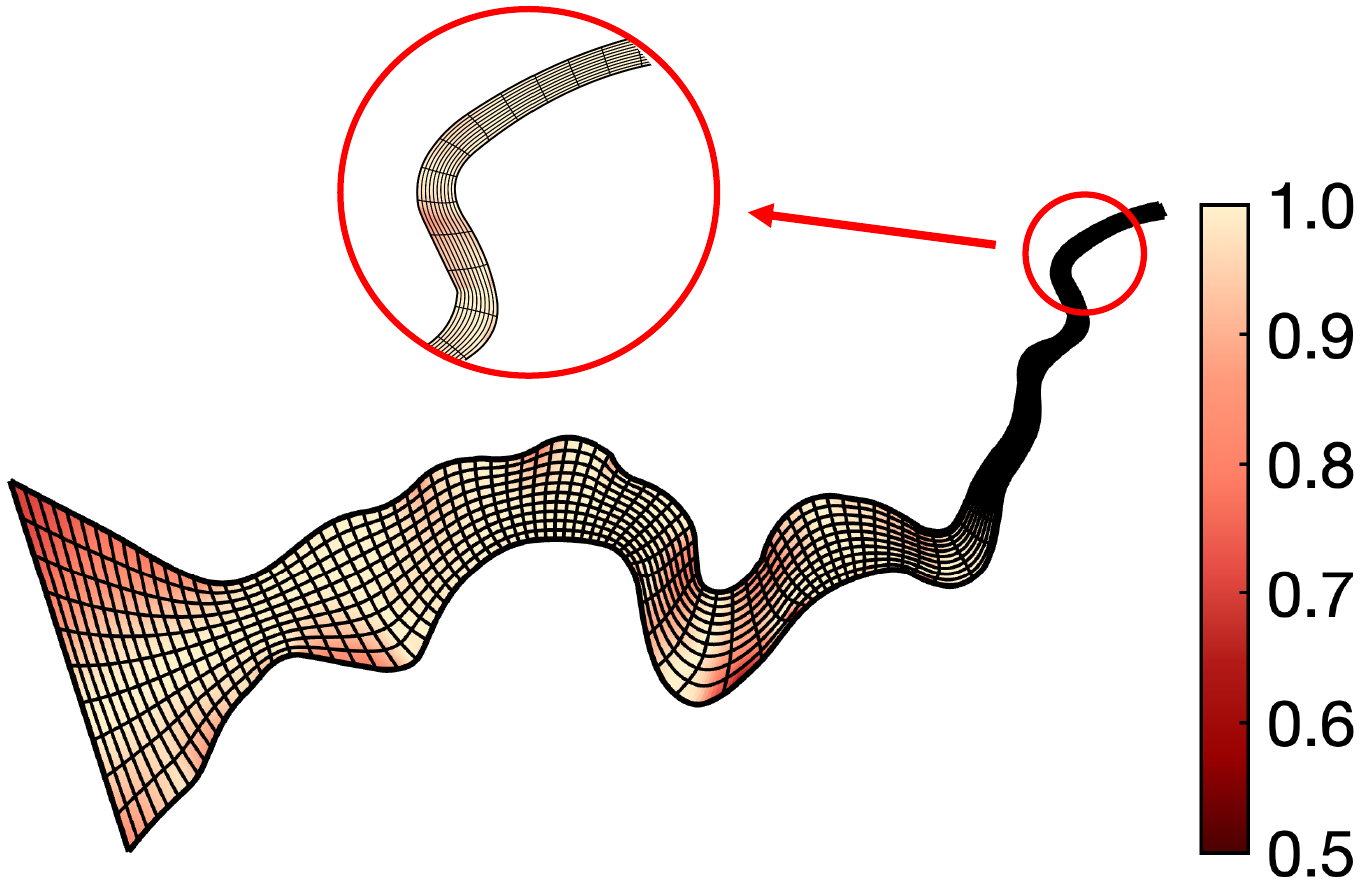}
    \quad
    \includegraphics[width=0.45\linewidth]{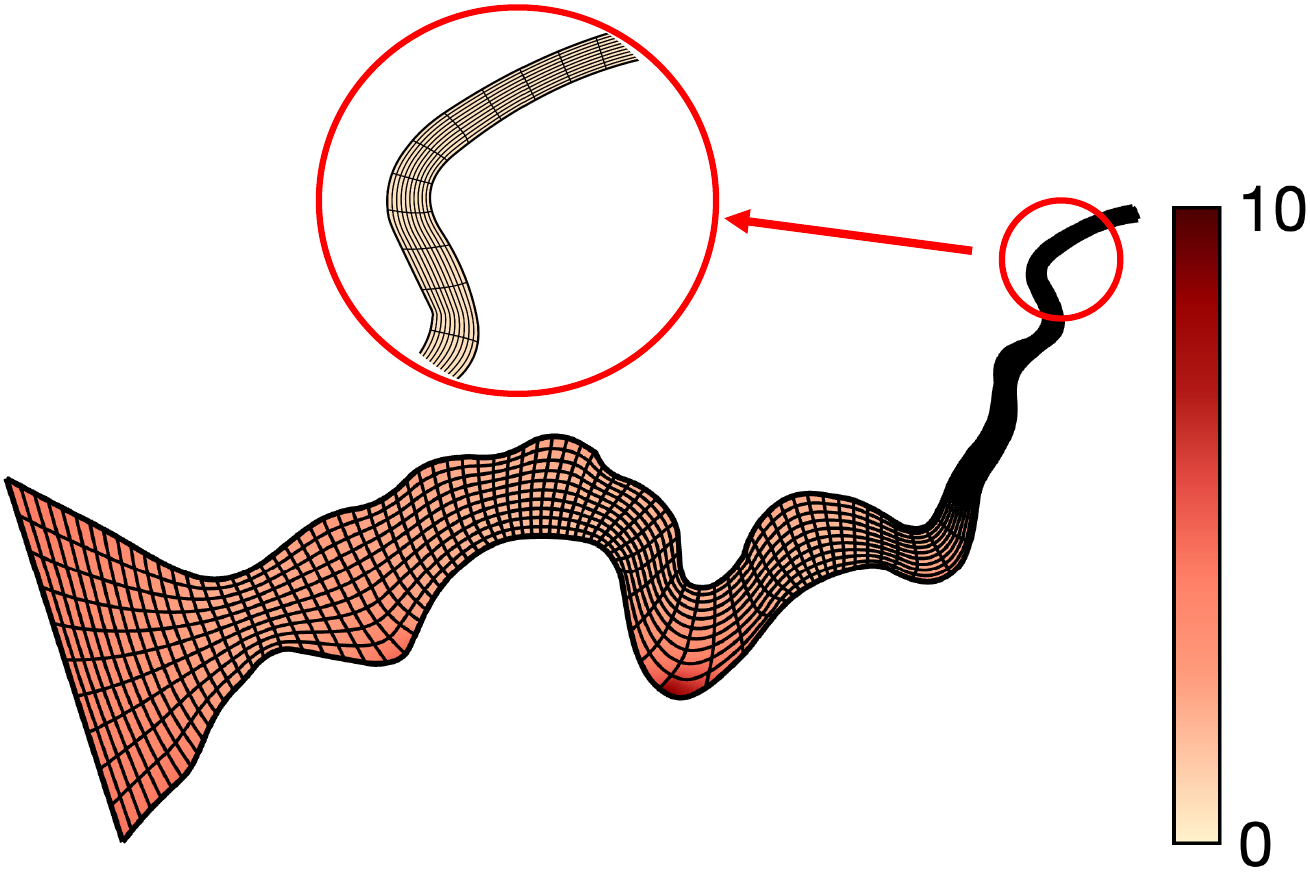}
    \label{fig6c:Scheldt_PDESJ}}
    \caption{Comparative analysis of parameterization techniques for the Scheldt estuary: Scaled Jacobian $\vert \bm{\mathcal{J}} \vert_s$ and uniformity metrics $m_{\rm unif.}$ across different methods.}
\label{fig6:Scheldt}
\end{figure}

Our proposed boundary parameter matching method addresses these issues effectively, enhancing the overall quality of parameterization. This improvement is achieved through a simple linear interpolation along the $\eta$-direction, as illustrated in Figure~\ref{fig6b:Scheldt_linearSJ}, leading to considerable enhancements. With the linear interpolation-based parameterization as an initial guess, further refinement is obtained by implementing our PDE-based parameterization technique, which results in a superior quality of parameterization, as shown in Figure~\ref{fig6c:Scheldt_PDESJ}. The effectiveness of these improvements is substantiated by comprehensive statistical data, detailed in Table~\ref{table1:data}. In this case, the initial application of the chord-length method results in an invalid parameterization, indicated by a negative minimum value of the scaled Jacobian $\vert \bm{\mathcal{J}} \vert_s$. This suggests that the parameterization is non-bijective. We omit the additional quality metrics for the non-bijective parameterization results.

\begin{table}[p]
  \caption{Comparisons between different methods: chord-length method, linear interpolation-based parameterization method, and PDE-based parameterization method. This table presents the minimum and average values of the scaled Jacobian $\vert \bm{\mathcal{J}} \vert_s$, along with the maximum and average values of the uniformity metric $m_{\rm unif.}$. Additionally, for the PDE-based method, which utilizes linear interpolation as an initial guess, the computational timings (in seconds) are also provided, marked with a plus (+). Results demonstrating the best performance are highlighted in \textbf{bold} for easy identification.}
  \centering
  \resizebox{\linewidth}{!}{
  \begin{tabular}{lllllll}
  \toprule
  \multirow{2}{*}{Model} & \multirow{2}{*}{Method} & \multicolumn{2}{c}{$\vert \bm{\mathcal{J}} \vert_s$} & \multicolumn{2}{c}{$m_{\rm unif.}$} & \multirow{2}{*}{Time (s)}\\
  \cline{3-4}
  \cline{5-6}
  & &min.&avg.&max.&avg.& \\
  \hline
  \multirow{3}{*}{Gulf of Mexico (Figure~\ref{fig1:boundary_correspondence_comparison})} & Chord-length method & $-1$ & - & - & - & - \\
  & Pan et al. \cite{pan2018low} & $-1$ & - & - & - & - \\
  & Zheng et al. \cite{zheng2019boundary} & $-1$ & - & - & - & - \\
  & Zhan et al. \cite{zhan2023boundary} & $-1$ & - & - & - & - \\
  & Zhan et al. \cite{zhan2024simultaneous} & $-1$ & - & - & - & - \\
  &Linear interpolation& $0.3035$ & $0.9802$ & $\mathbf{5.9168}$ & $2.0606$ & $0.6565$ \\
  &PDE-based method& $\mathbf{0.4577}$ & $\mathbf{0.9843}$ & $9.6376$ & $\mathbf{2.0589}$ & $+0.0378$ \\
  \hline
  \multirow{3}{*}{Scheldt estuary (Figure~\ref{fig6:Scheldt}) } & Chord-length method & $-0.0053$ & - & - & - & - \\
  & Pan et al. \cite{pan2018low} & $-1$ & - & - & - & - \\
  & Zheng et al. \cite{zheng2019boundary} & $-1$ & - & - & - & - \\
  & Zhan et al. \cite{zhan2023boundary} & $-1$ & - & - & - & - \\
  & Zhan et al. \cite{zhan2024simultaneous} & $-1$ & - & - & - & - \\
  &Linear interpolation& $\mathbf{0.6410}$ & $\mathbf{0.96821}$ & $\mathbf{8.0141}$ & $\mathbf{2.1636}$ & $1.2570$ \\
  &PDE-based method& $0.3727$ & $0.9655$ & $9.4963$ & $2.1373$ & $+0.0617$\\
  \hline
  \multirow{2}{*}{Letter ``M'' (Figure~\ref{fig7:GMP})} &Chord-length method& $0.2822$ & $0.8587$ & $\mathbf{2.4644}$ & $\mathbf{2.0101}$ & - \\
  & Pan et al. \cite{pan2018low} & $-1$ & - & - & - & - \\
  & Zheng et al. \cite{zheng2019boundary} & $-1$ & - & - & - & - \\
  & Zhan et al. \cite{zhan2023boundary} & $-1$ & - & - & - & - \\
  & Zhan et al. \cite{zhan2024simultaneous} & $-1$ & - & - & - & - \\
  & Linear interpolation & $\mathbf{0.6969}$ & $\mathbf{0.9955}$ & $11.8670$ & $2.0501$ & $0.4699$ \\
  \hline
  \multirow{2}{*}{Letter ``G'' (Figure~\ref{fig7:letterG})} & Chord-length method & $0.43848$ & $0.8866$ & $\mathbf{2.8831}$ & $\mathbf{2.0186}$ & -\\
  & Pan et al. \cite{pan2018low} & $-1$ & - & - & - & - \\
  & Zheng et al. \cite{zheng2019boundary} & $-1$ & - & - & - & - \\
  & Zhan et al. \cite{zhan2023boundary} & $-1$ & - & - & - & - \\
  & Zhan et al. \cite{zhan2024simultaneous} & $-1$ & - & - & - & - \\
  &Linear interpolation& $\mathbf{0.69859}$ & $\mathbf{0.9873}$ & $3.3872$ & $2.0192$ & $0.4144$ \\
  \hline
  \multirow{2}{*}{Letter ``P'' (Figure~\ref{fig8:letterP})} &Chord-length method& $-0.3541$ & - & - & - & - \\
  & Pan et al. \cite{pan2018low} & $-1$ & - & - & - & - \\
  & Zheng et al. \cite{zheng2019boundary} & $-1$ & - & - & - & - \\
  & Zhan et al. \cite{zhan2023boundary} & $-1$ & - & - & - & - \\
  & Zhan et al. \cite{zhan2024simultaneous} & $-1$ & - & - & - & - \\
  & Linear interpolation & $\mathbf{0.91606}$ & $\mathbf{0.9929}$ & $\mathbf{3.9123}$ & $\mathbf{2.0424}$ & $0.3922$ \\
  \hline
  \multirow{2}{*}{Channel (Figure~\ref{fig8:Channel})} & Chord-length method & $0.3856$ & $0.7828$ & $2.4072$ & $2.0217$ & -\\
  & Pan et al. \cite{pan2018low} & $-1$ & - & - & - & - \\
  & Zheng et al. \cite{zheng2019boundary} & $-1$ & - & - & - & - \\
  & Zhan et al. \cite{zhan2023boundary} & $-1$ & - & - & - & - \\
  & Zhan et al. \cite{zhan2024simultaneous} & $-1$ & - & - & - & - \\
  &Linear interpolation& $\mathbf{0.66439}$ & $\mathbf{0.9838}$ & $\mathbf{7.3417}$ & $\mathbf{2.0666}$ & $0.4285$ \\
\bottomrule
\end{tabular}}
\label{table1:data}
\end{table}

\begin{figure}[H]
    \centering
    \subfigure[Initial chord-length parameterization]{\includegraphics[width=0.42\linewidth]{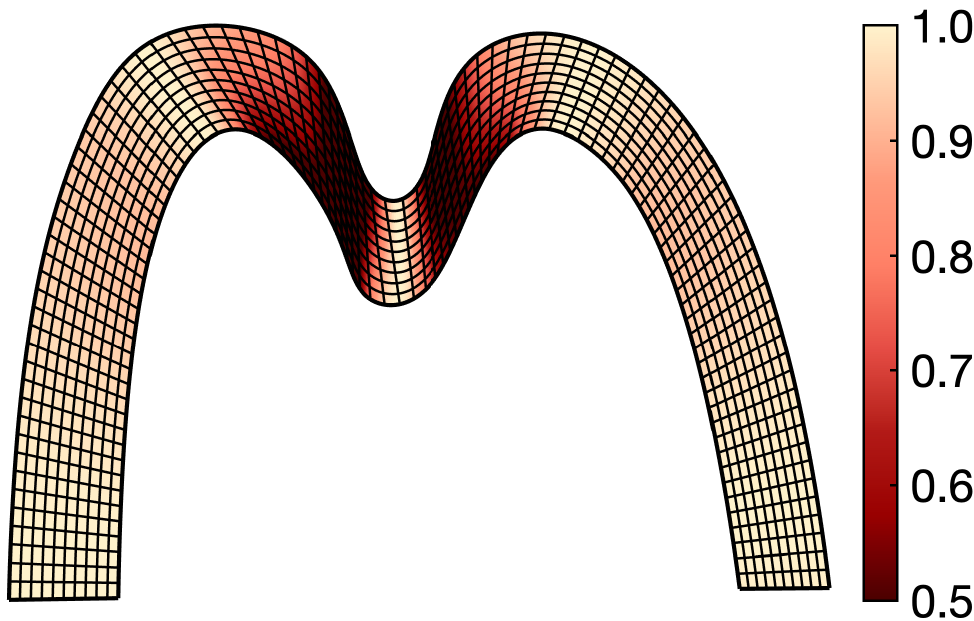}
    \label{subfig:initial}
    } \quad
    \subfigure[Linear interpolation-based parameterization with the proposed boundary parameter matching method]{
    \includegraphics[width=0.42\linewidth]{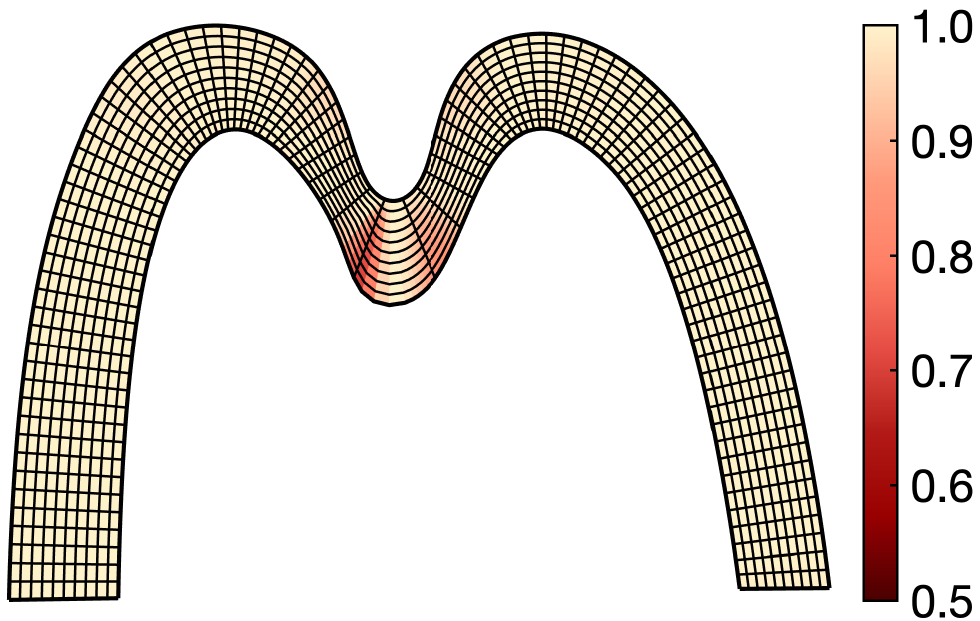}
    \label{subfig:linear}
    }
    \caption{Scaled Jacobian $\vert \bm{\mathcal{J}} \vert_s$ for a letter ``M''-shaped domain.}
    \label{fig7:GMP}
\end{figure}

Building on the previous discussion, the effectiveness of our proposed method is further exemplified with additional ``M''-shaped domain, as shown in Figure~\ref{fig7:GMP}. This example is instrumental in highlighting the substantial improvement in parameterization quality achieved by our approach. Remarkably, even a straightforward linear interpolation technique, when applied in conjunction with our boundary parameter approach, results in parameterizations of surprisingly high quality. In instances like these, resorting to a PDE-based parameterization method to further improve quality becomes superfluous. This outcome demonstrates not only the robustness of our method but also its capability and robustness in handling complex geometries.

\subsection{Comparisons with existing methods}
\label{sec503:comparison}

This section presents a comparative analysis of our proposed method against established techniques, specifically the Low-rank parameterization method introduced by Pan et al. \cite{pan2018low}, the boundary correspondence method utilizing optimal mass transport as proposed by Zheng et al. \cite{zheng2019boundary}, the boundary corners selection method based on deep learning by Zhan et al. \cite{zhan2023boundary}, and the simultaneous boundary and interior parameterization method via deep learning proposed by Zhan et al.\cite{zhan2024simultaneous}.

For our comparative analysis, we utilize the open-source implementation of the boundary correspondence method provided by the original authors \cite{zheng2019boundary}, which is accessible at \url{https://github.com/ZH-ye/BoundaryCorrespondenceOT}. Within the workflow of this method, the authors first identify optimal corner points that align the curvature measure of the physical domain $\Omega$ with that of the unit square. Subsequently, a spline-based least-squares fitting subroutine in conjunction with chord-length parameterization is utilized to approximate the input point cloud, which is followed by the integration of the low-rank quasi-conformal mapping parameterization method \cite{pan2018low} to complete. While the last two competitors by Zhan et al. from \cite{zhan2023boundary, zhan2024simultaneous}, the parameterizations are kindly provided by the original authors.

\begin{figure}[H]
    \centering
    \subfigure[Zheng et al. \cite{zheng2019boundary}]{
    \includegraphics[width=0.3\linewidth]{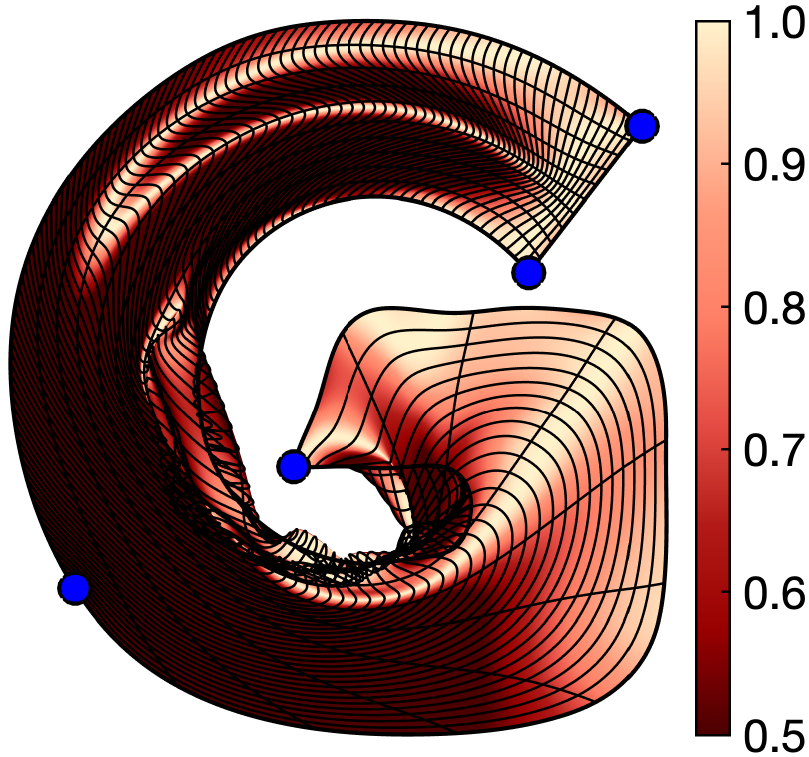}
    \label{fig7a:zheng}}\quad
    \subfigure[Pan et al. \cite{pan2018low}]{ \includegraphics[width=0.3\linewidth]{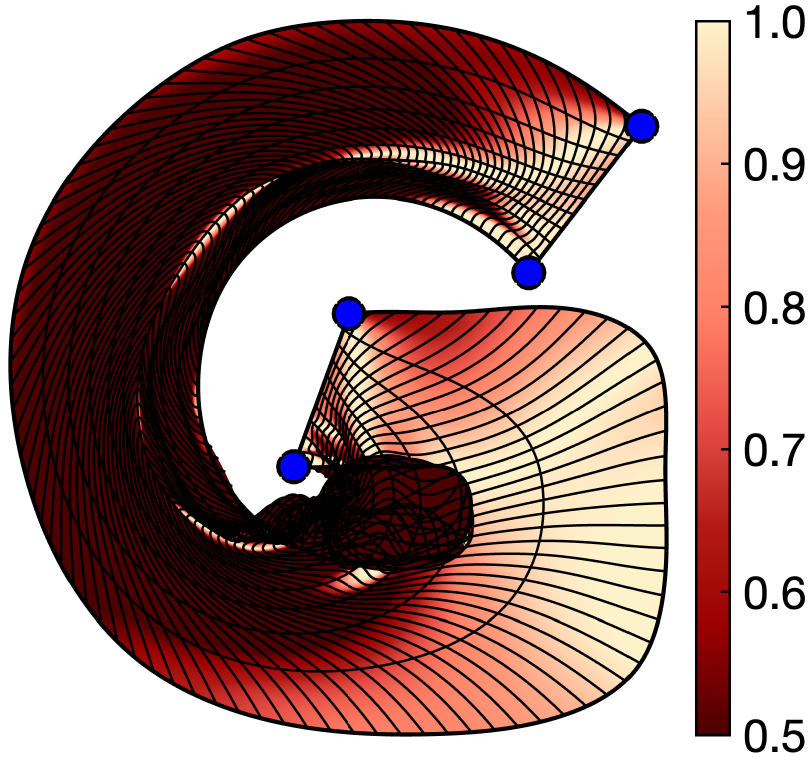}
    \label{fig7b:pan}}\quad
    \subfigure[Zhan et al. \cite{zhan2023boundary}]{
    \includegraphics[width=0.3\linewidth]{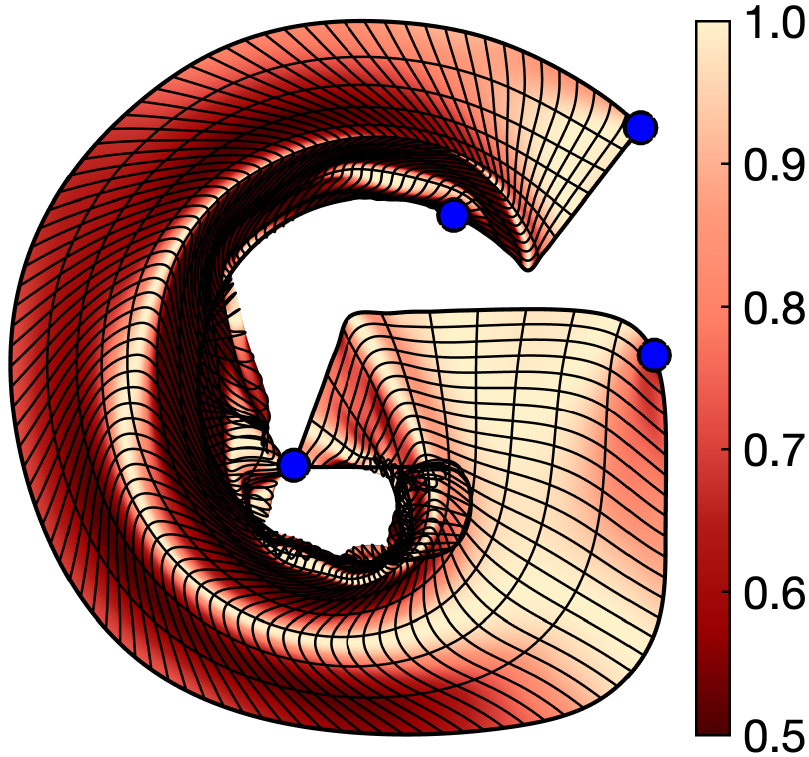}
    \label{fig8c:zhan2023}}\\
    \subfigure[Zhan et al. \cite{zhan2024simultaneous}]{
    \includegraphics[width=0.3\linewidth]{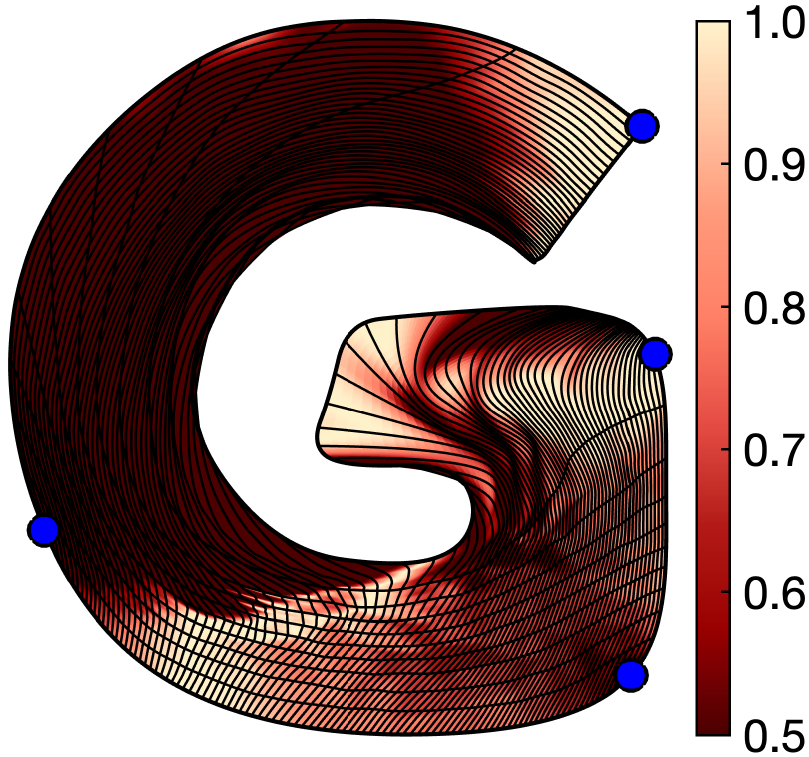}
    \label{fig8d:zhan2024}}\quad
    \subfigure[Initial chord-length parameterization]{
    \includegraphics[width=0.3\linewidth]{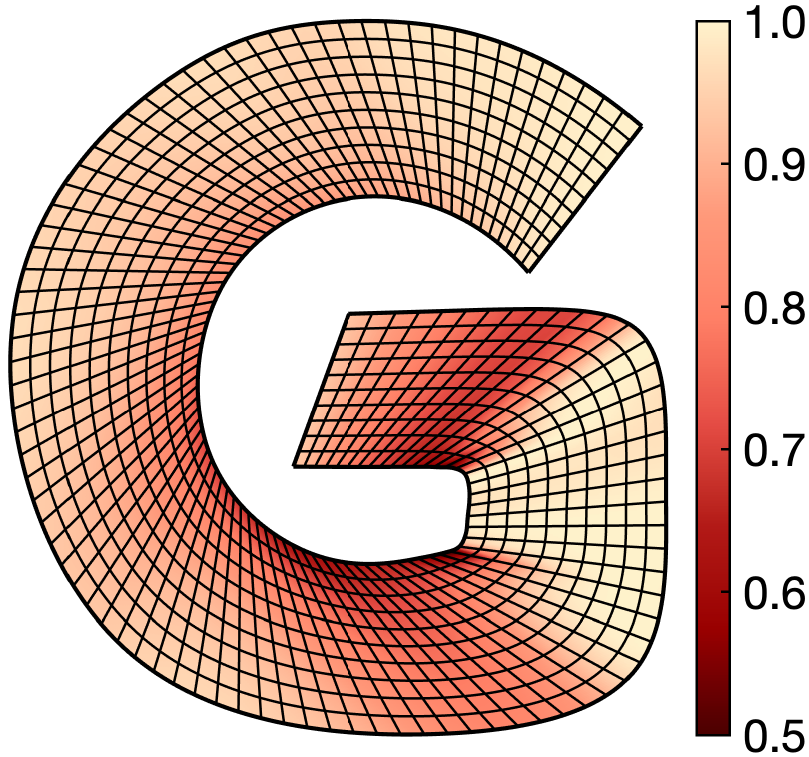}
    \label{fig7c:initial}}\quad
    \subfigure[Proposed method]{
    \includegraphics[width=0.3\linewidth]{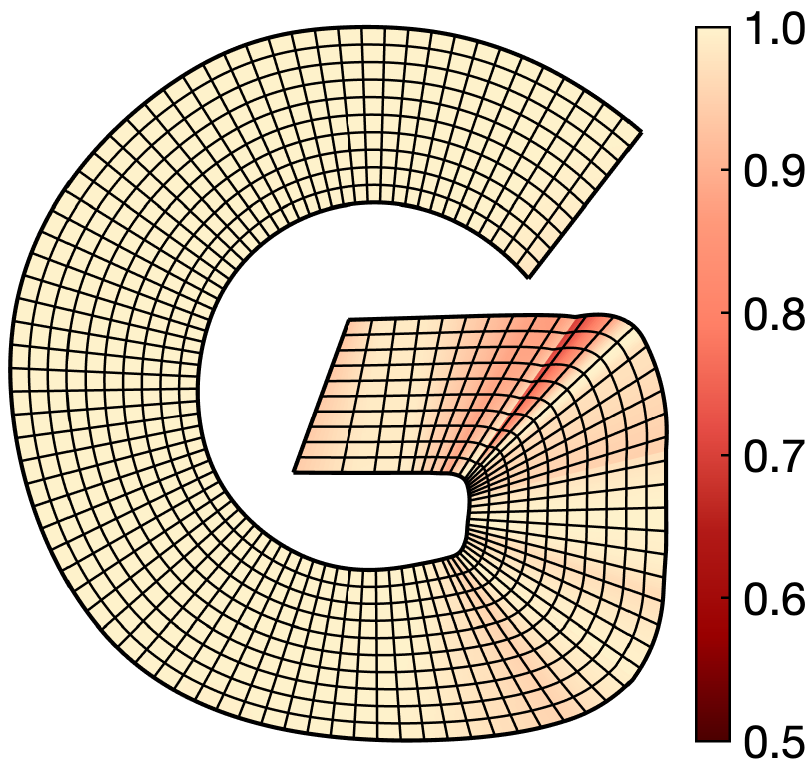}
    \label{fig7d:ours}}
    \caption{Letter ``G'': Comparison of Scaled Jacobian ($\vert \bm{\mathcal{J}} \vert_s$) with different methods}
    \label{fig7:letterG}
\end{figure}

Figure~\ref{fig7:letterG} shows the parameterizations of a "G"-shaped domain obtained through various existing methods alongside our proposed approach. As depicted in Figure~\ref{fig7a:zheng}, the boundary correspondence method \cite{zheng2019boundary} does not correctly identify the corners in this example. Specifically, the least-square approximation of the input curves results in the loss of geometric precision at a sharp corner, an issue our method avoids by preserving the exact geometry of the boundaries. A key concern is the significant distortion in the resulting parameterization due to the mismatched opposite boundaries. To address this, we establish correct corner correspondences and engage the low-rank quasi-conformal method \cite{pan2019low}. Figure~\ref{fig8c:zhan2023} shows the automatic corner detection using a deep learning method as proposed by Zhan et al. \cite{zhan2023boundary}, coupled with an interior parameterization achieved through the low-rank quasi-conformal method detailed by Pan et al. \cite{pan2018low}. It is evident that the parameterization quality is poor due to the mismatching of the opposing boundaries. Figure~\ref{fig9d:zhan2024} demonstrates the resulting non-bijective parameterization characterized by numerous self-intersections, as produced by the deep learning approach \cite{zhan2024simultaneous}. However, as shown in Figure~\ref{fig7b:pan}, this approach yields a non-bijective parameterization (refer to Table~\ref{table1:data} please). In contrast, parameterizations initiated with chord-length parameterization and further refined using our method with simple linear interpolation produce bijective outcomes, as demonstrated in Figure~\ref{fig7c:initial} and Figure~\ref{fig7d:ours}, respectively. Notably, our method significantly enhances the orthogonality of the resulting parameterization compared with the initial parameterization in Figure~\ref{fig7c:initial}.

\begin{figure}[H]
    \centering
    \subfigure[Zheng et al. \cite{zheng2019boundary}]{ 
    \includegraphics[width=0.3\linewidth]{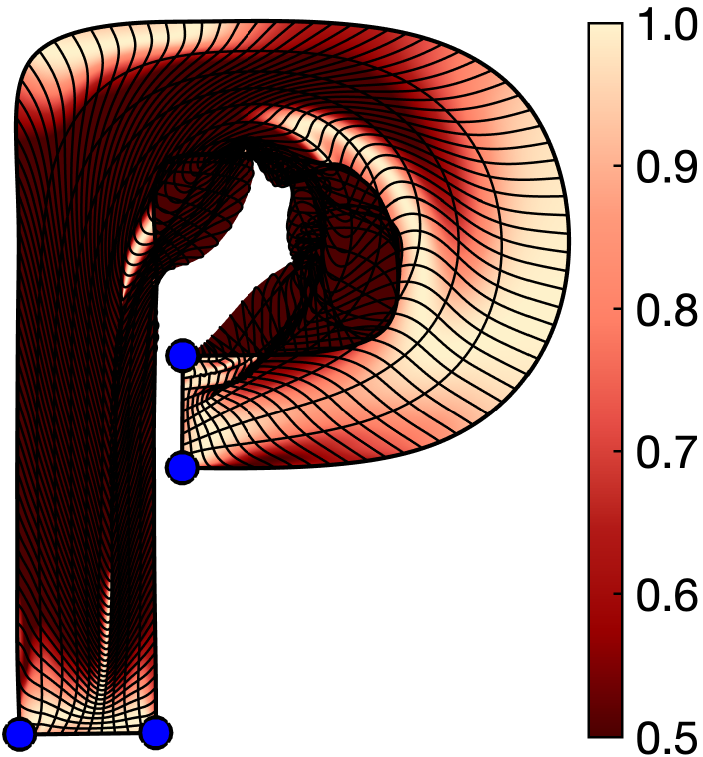}
    \label{fig9a:zheng}}\quad
    \subfigure[Pan et al. \cite{pan2018low}]{ \includegraphics[width=0.3\linewidth]{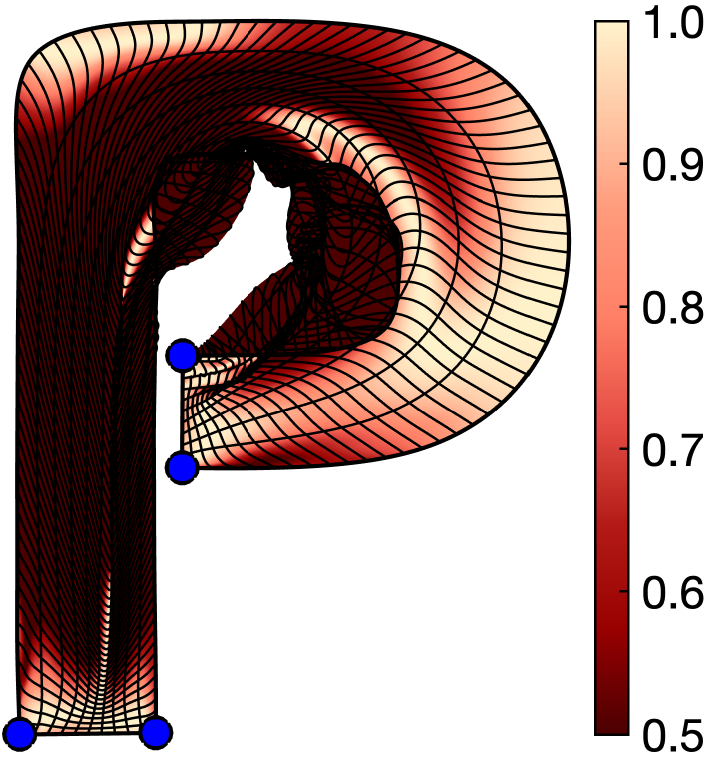}
    \label{fig9b:pan}}\quad
    \subfigure[Zhan et al. \cite{zhan2023boundary}]{
    \includegraphics[width=0.3\linewidth]{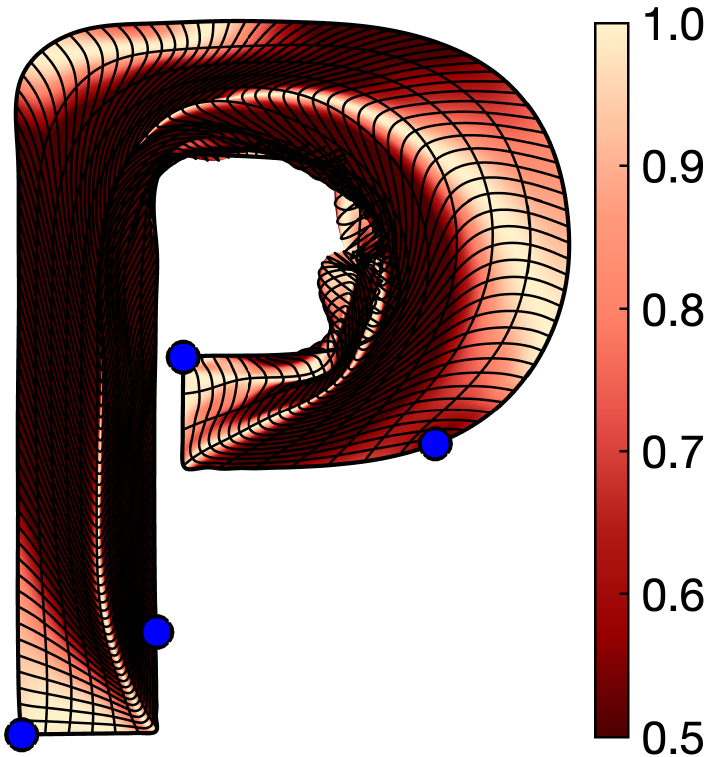}
    \label{fig9c:zhan2023}}\\
    \subfigure[Zhan et al. \cite{zhan2024simultaneous}]{
    \includegraphics[width=0.3\linewidth]{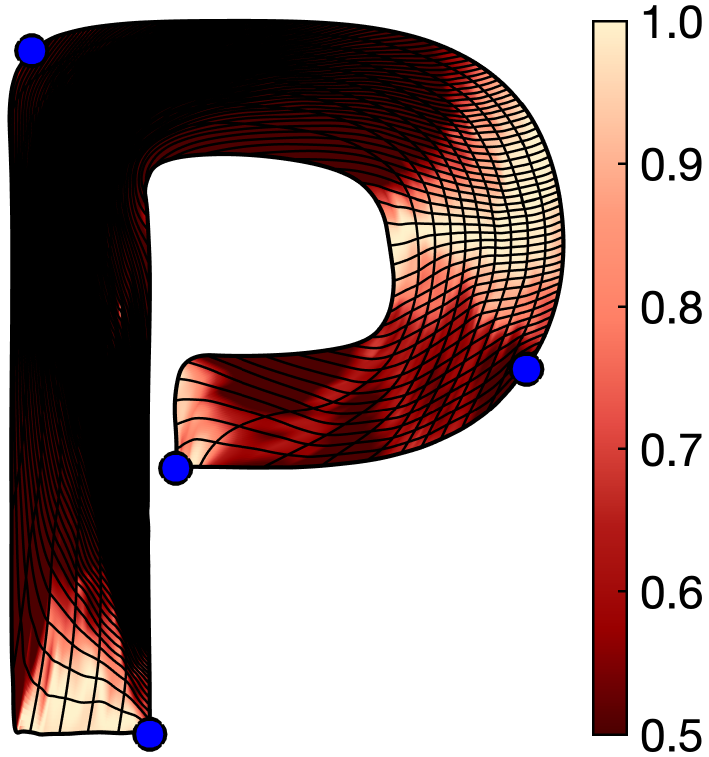}
    \label{fig9d:zhan2024}}\quad
    \subfigure[Initial chord-length parameterization]{
    \includegraphics[width=0.3\linewidth]{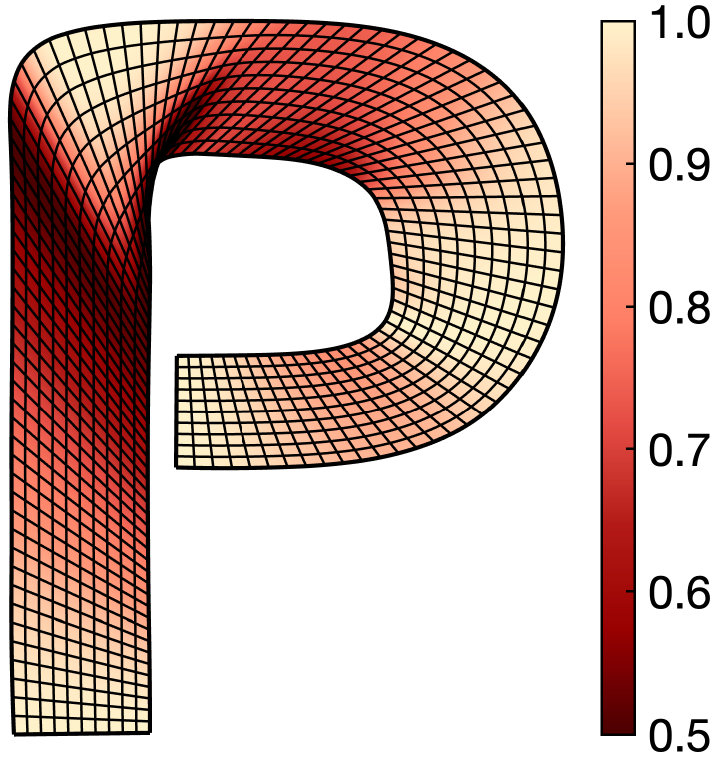}
    \label{fig9e:initial}}\quad
    \subfigure[Proposed method]{
    \includegraphics[width=0.3\linewidth]{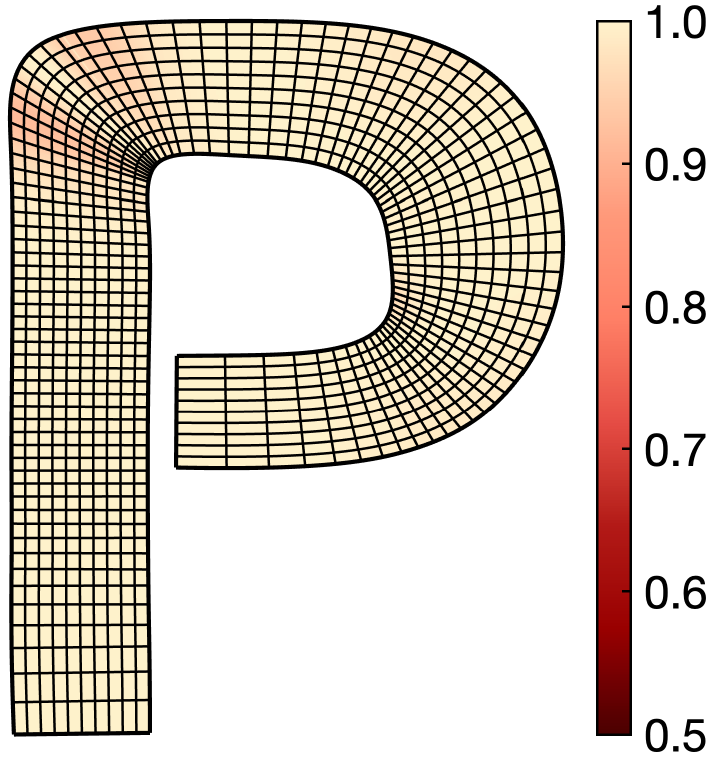}
    \label{fig9f:ours}}
    \caption{Letter ``P'': Comparison of Scaled Jacobian ($\vert \bm{\mathcal{J}} \vert_s$) with different methods}
    \label{fig8:letterP}
\end{figure}

Additional comparative results are presented in Figure~\ref{fig8:letterP}. In this case, the boundary correspondence method \cite{zheng2019boundary}, as depicted in Figure~\ref{fig9a:zheng}, successfully identifies the four corners, leading to a parameterization result comparable to that achieved by the low-rank quasi-conformal method (Figure~\ref{fig9b:pan}). Figure~\ref{fig9c:zhan2023} and Figure~\ref{fig9d:zhan2024} show the resulting non-bijective parameterizations generated by \cite{zhan2023boundary} and \cite{zhan2024simultaneous}, respectively. The parameterization derived from the initial chord-length method, illustrated in Figure~\ref{fig9e:initial}, exhibits significant distortion, particularly at the upper left turn. Comprehensive statistical analysis is provided in Table~\ref{table1:data}, where the negative minimum value of the scaled Jacobian $\left \vert \mathcal{J} \right \vert_s$ indicates that these parameterizations are non-bijective. In contrast, our proposed parameterization method, shown in Figure~\ref{fig9f:ours}, demonstrates superior parameterization is achieved.

\section{Applications}
\label{sec6:applications}

\subsection{Application to solid modeling}
\label{sec601:extrudedvolume}

In $3$D solid modeling applications, numerous CAD models can be generated or effectively simplified to planar cases through fundamental operations such as extrusion, sweeping, lofting, ruling, and revolving \cite{piegl1996nurbs}. Although this paper primarily focuses on planar domains, the proposed method can be readily adapted to $3$D volumes created by using these fundamental modeling operations. Figure~\ref{fig10:solid} shows volumetric parameterizations generated by combining the proposed method with extrusion and sweeping, showcasing its applicability to solid modeling.

\begin{figure}[H]
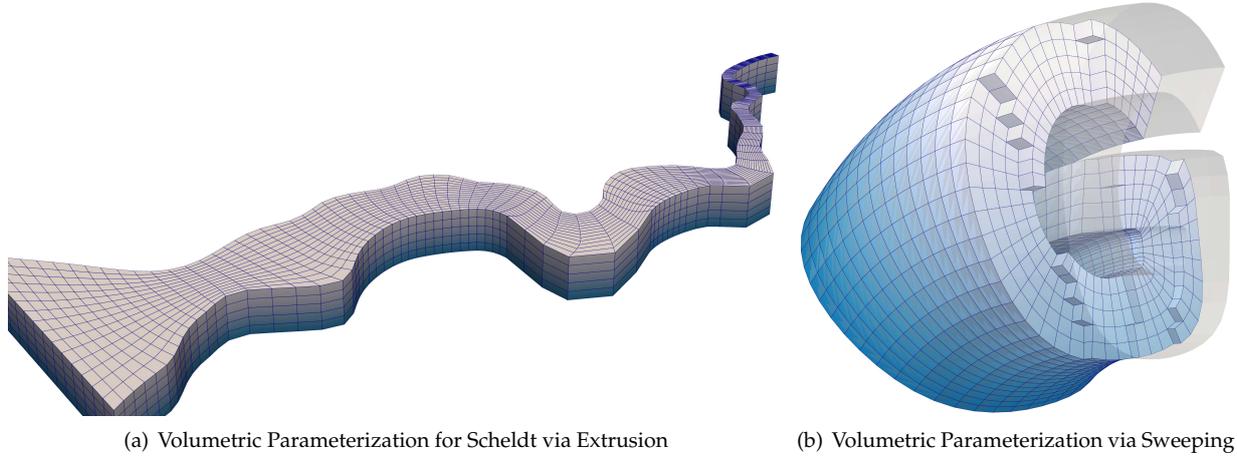

\centering
\subfigure[Volumetric Parameterization for Scheldt via Extrusion]{\includegraphics[width=0.62\linewidth]{figures/fig10a_extrudeScheldt.pdf}} \
\subfigure[Volumetric Parameterization via Sweeping]{\includegraphics[width=0.35\linewidth]{figures/fig10b_sweepG.pdf}}
\caption{Illustration of solid models generated by integrating the proposed parameterization method with standard solid modeling operations, e.g., extrusion and sweeping.}
\label{fig10:solid}
\end{figure}

\subsection{Application to IGA simulation}
\label{sec602:iga}

To validate the usability of our parameterizations for IGA simulation, we consider Poisson’s problem with mixed boundary conditions in the domain $\Omega$ (see Figure~\ref{fig8:Channel}):
\begin{equation}
\left \{
\begin{aligned}
     -\Delta u&= f,\ \ \text{in}\ \Omega,\\ 
     u&= g,\ \ \text{on}\ \Gamma_{D},\\
     \nabla u \cdot \mathbf{n} &=h,\ \ \text{on}\ \Gamma_{N},\\
\end{aligned}
\right.
\label{eq:7101}
\end{equation}
where $f \in L^2(\Omega): \Omega \rightarrow \mathbb{R}$ denotes a given source term, $\partial\Omega=\bar{\Gamma}_{D}\cup \bar{\Gamma}_{N}$ defines the boundary of the physical domain $\Omega$ with $\Gamma_{D}\cap \Gamma_{N}=\mathbf{\varnothing}$, $\mathbf{n}$ denotes the outward pointing unit normal vector on the boundary $\partial \Omega$, $\Gamma_{D}$ and $\Gamma_{N}$ denote the separate parts of the boundary where Dirichlet and Neumann boundary conditions are prescribed, respectively.

\begin{figure}[H]
    \centering
    \subfigure[Initial chord-length parameterization]{\includegraphics[width=0.45\linewidth]{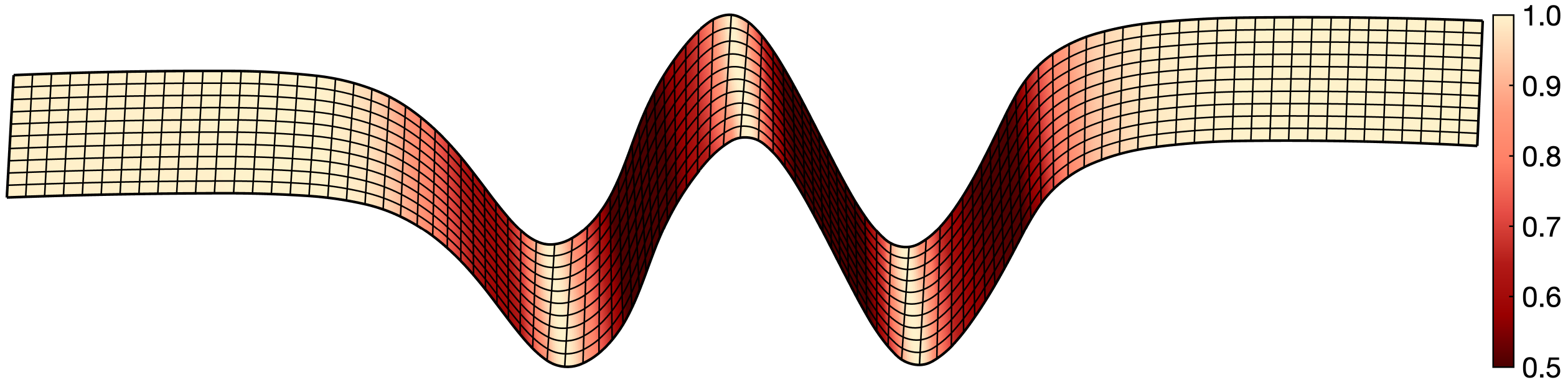}
    \label{fig8a:initial}}
    \quad 
    \subfigure[Linear interpolation-based parameterization using the proposed boundary parameter matching method]{\includegraphics[width=0.45\linewidth]{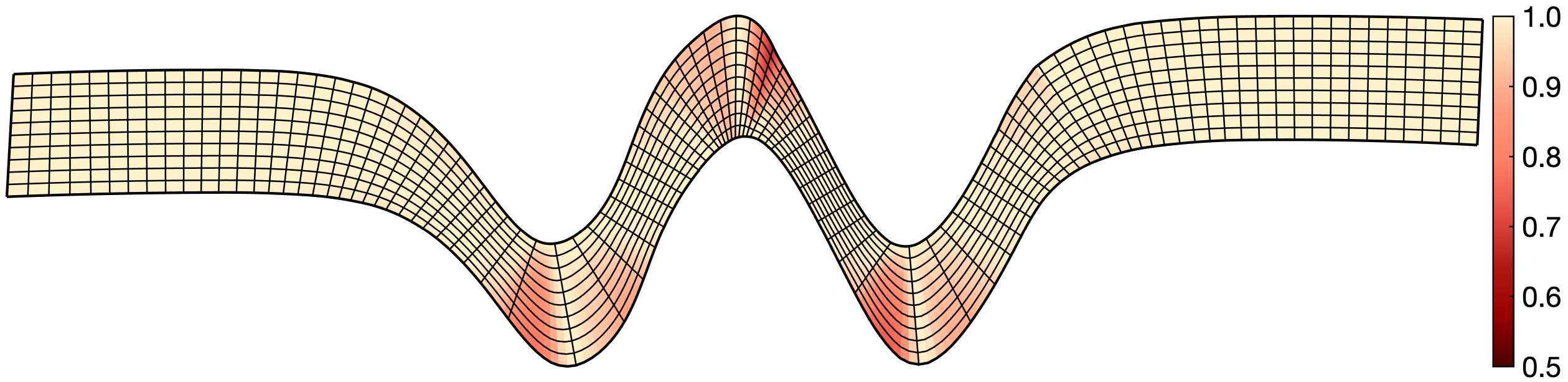}
    \label{fig8b:linear}} \\
    \subfigure[Boundary condition settings]{\includegraphics[width=0.46\linewidth]{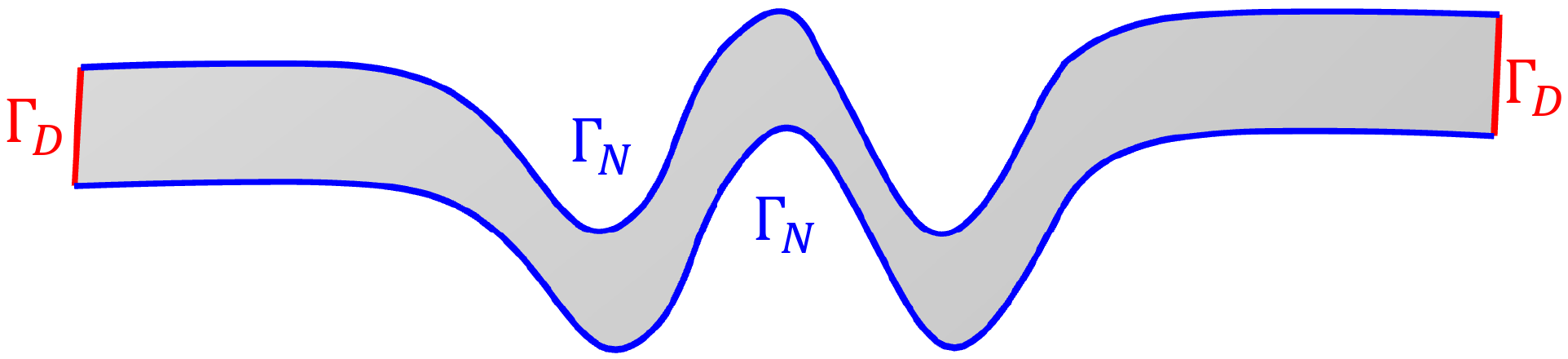}
    \label{fig8c:boundary conditions}}
    \quad 
    \subfigure[Exact solution $u^{\rm exact}$]{\includegraphics[width=0.45\linewidth]{figures/fig11d_channel_exact_solution.pdf}
    \label{fig8d:exact solution}} 
    \caption{IGA simulation on a channel geometry.}
    \label{fig8:Channel}
\end{figure}

Figures~\ref{fig8a:initial} and~\ref{fig8b:linear} present the initial chord-length and linear interpolation-based parameterizations achieved through our boundary matching method. These visualizations, with the scaled Jacobian $\vert \bm{\mathcal{J}} \vert_s$ color-encoded, demonstrate significant improvement in orthogonality due to our method. For this example, the source term is set as $f = 8 \pi^2 \sin(2 \pi x) \sin(2 \pi y)$, leading to the exact solution $u^{\rm exact} = \sin(2 \pi x) \sin(2 \pi y)$. The boundary conditions and the exact solution are depicted in Figures~\ref{fig8c:boundary conditions} and~\ref{fig8d:exact solution}, respectively.

\begin{figure}[H]
    \centering
    \includegraphics[width=0.85\linewidth]{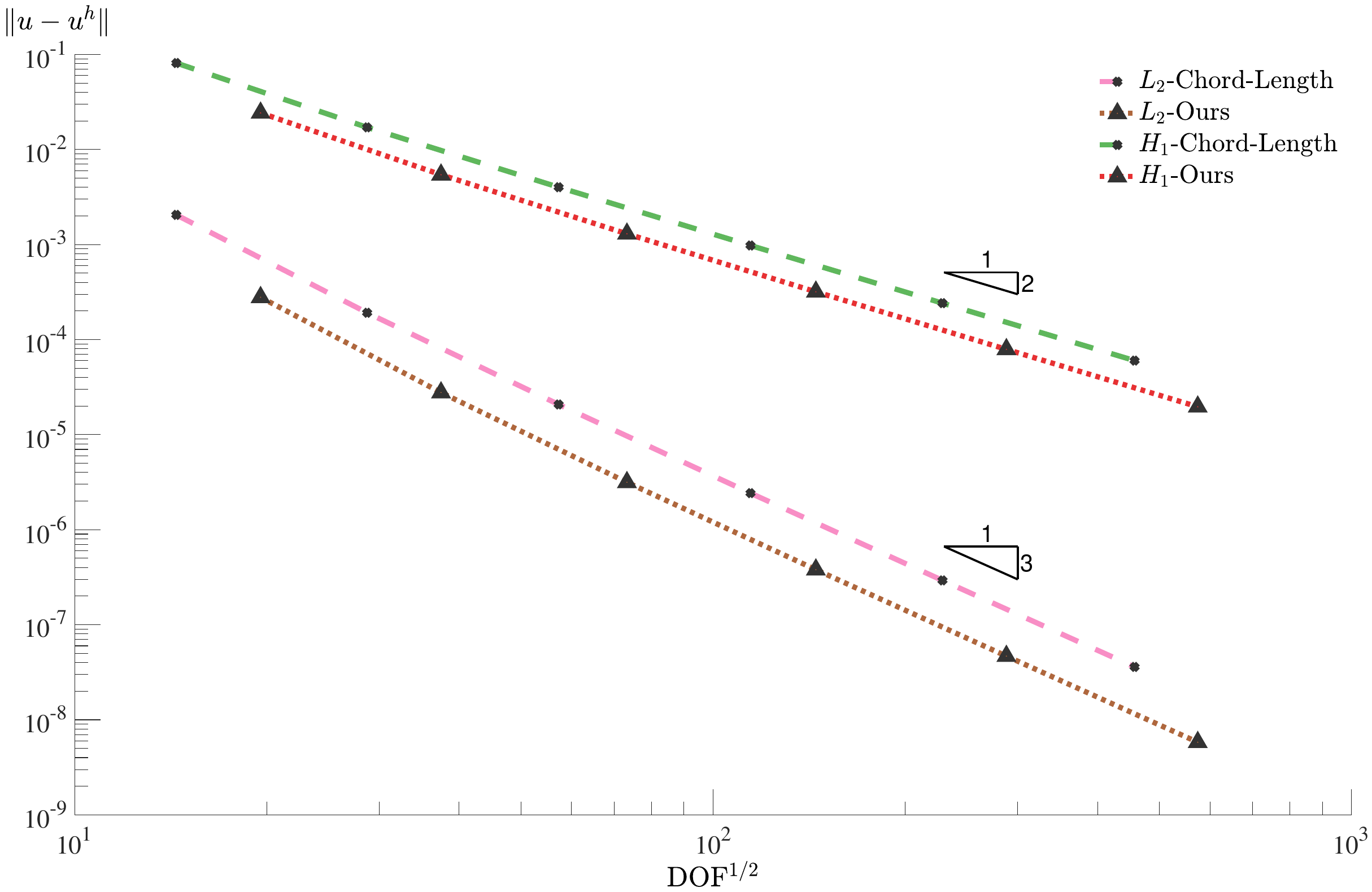}
    \caption{Error convergence during $h$-refinement.}
    \label{fig9:convergencePlot}
\end{figure}

Figure~\ref{fig9:convergencePlot} shows the $L_2$ and $H_1$ error convergence for different parameterizations. The results highlight a substantial reduction in error for our improved parameterization compared to the original chord-length parameterization, with all methods achieving the optimal rate of convergence.

\section{Conclusions and outlook}
\label{sec7:conclusions}

This paper presents a novel method designed to address the boundary parameter correspondence challenge, a crucial aspect of analysis-suitable parameterization in IGA. Our approach, integrating Schwarz-Christoffel mapping with a subsequent curve reparameterization procedure, effectively maintains the geometric exactness and continuity of the input. Significantly, our boundary parameter matching technique enables even basic linear interpolation methods to yield high-quality parameterizations, especially in elongated domains. Through numerous numerical experiments, the effectiveness and reliability of our method have been convincingly demonstrated, underscoring its potential in practical applications.

Despite its strengths, our method is not without limitations. A significant constraint lies in the planar nature of conformal mapping, while the demand for mesh generation predominantly exists in three-dimensional contexts. This discrepancy presents a challenge and an opportunity for future research. Exploring the potential of extending our method to address surface parameter correspondence in three-dimensional spaces is an exciting and valuable direction for ongoing investigations. Such advancements could significantly broaden the applicability and impact of our approach in the field of IGA.

\section*{Acknowledgements}
\label{sec:Acknowledgement}

This research was partially funded by the National Natural Science Foundation of China (Grant No.~12001287).

\bibliographystyle{elsarticle-num}
\bibliography{mybibfile}

\begin{thebibliography}{10}
\expandafter\ifx\csname url\endcsname\relax
  \def\url#1{\texttt{#1}}\fi
\expandafter\ifx\csname urlprefix\endcsname\relax\def\urlprefix{URL }\fi
\expandafter\ifx\csname href\endcsname\relax
  \def\href#1#2{#2} \def\path#1{#1}\fi

\bibitem{hughes2005isogeometric}
T.~J. Hughes, J.~A. Cottrell, Y.~Bazilevs, {Isogeometric analysis: CAD, finite elements, NURBS, exact geometry and mesh refinement}, Computer Methods in Applied Mechanics and Engineering 194~(39-41) (2005) 4135--4195.

\bibitem{cottrell2009isogeometric}
J.~A. Cottrell, T.~J. Hughes, Y.~Bazilevs, {Isogeometric analysis: Toward integration of CAD and FEA}, John Wiley \& Sons, 2009.

\bibitem{cohen2010analysis}
E.~Cohen, T.~Martin, R.~Kirby, T.~Lyche, R.~Riesenfeld, {Analysis-aware modeling: Understanding quality considerations in modeling for isogeometric analysis}, Computer Methods in Applied Mechanics and Engineering 199~(5-8) (2010) 334--356.

\bibitem{xu2013optimal}
G.~Xu, B.~Mourrain, R.~Duvigneau, A.~Galligo, {Optimal analysis-aware parameterization of computational domain in 3D isogeometric analysis}, Computer-Aided Design 45~(4) (2013) 812--821.

\bibitem{pilgerstorfer2014bounding}
E.~Pilgerstorfer, B.~J{\"u}ttler, {Bounding the influence of domain parameterization and knot spacing on numerical stability in Isogeometric Analysis}, Computer Methods in Applied Mechanics and Engineering 268 (2014) 589--613.

\bibitem{xu2011parameterization}
G.~Xu, B.~Mourrain, R.~Duvigneau, A.~Galligo, {Parameterization of computational domain in isogeometric analysis: Methods and comparison}, Computer Methods in Applied Mechanics and Engineering 200~(23-24) (2011) 2021--2031.

\bibitem{xu2014high}
G.~Xu, B.~Mourrain, A.~Galligo, T.~Rabczuk, {High-quality construction of analysis-suitable trivariate NURBS solids by reparameterization methods}, Computational Mechanics 54~(5) (2014) 1303--1313.

\bibitem{hinz2018elliptic}
J.~Hinz, M.~M{\"o}ller, C.~Vuik, {Elliptic grid generation techniques in the framework of isogeometric analysis applications}, Computer Aided Geometric Design 65 (2018) 48--75.

\bibitem{pan2019isogeometric}
Q.~Pan, T.~Rabczuk, G.~Xu, C.~Chen, {Isogeometric analysis for surface PDEs with extended Loop subdivision}, Journal of Computational Physics 398 (2019) 108892.

\bibitem{liu2020simultaneous}
H.~Liu, Y.~Yang, Y.~Liu, X.-M. Fu, {Simultaneous interior and boundary optimization of volumetric domain parameterizations for IGA}, Computer Aided Geometric Design 79 (2020) 101853.

\bibitem{ji2023improved}
Y.~Ji, K.~Chen, M.~M{\"o}ller, C.~Vuik, {On an improved PDE-based elliptic parameterization method for isogeometric analysis using preconditioned Anderson acceleration}, Computer Aided Geometric Design 102 (2023) 102191.

\bibitem{pan2023g1}
M.~Pan, R.~Zou, W.~Tong, Y.~Guo, F.~Chen, G1-smooth planar parameterization of complex domains for isogeometric analysis, Computer Methods in Applied Mechanics and Engineering 417 (2023) 116330.

\bibitem{zheng2019boundary}
Y.~Zheng, M.~Pan, F.~Chen, Boundary correspondence of planar domains for isogeometric analysis based on optimal mass transport, Computer-Aided Design 114 (2019) 28--36.

\bibitem{farin1999discrete}
G.~Farin, D.~Hansford, {Discrete Coons patches}, Computer Aided Geometric Design 16~(7) (1999) 691--700.

\bibitem{gravesen2012planar}
J.~Gravesen, A.~Evgrafov, D.-M. Nguyen, P.~N{\o}rtoft, {Planar parametrization in isogeometric analysis}, in: International Conference on Mathematical Methods for Curves and Surfaces, Springer, 2012, pp. 189--212.

\bibitem{wang2014optimization}
X.~Wang, X.~Qian, {An optimization approach for constructing trivariate B-spline solids}, Computer-Aided Design 46 (2014) 179--191.

\bibitem{pan2020volumetric}
M.~Pan, F.~Chen, W.~Tong, {Volumetric spline parameterization for isogeometric analysis}, Computer Methods in Applied Mechanics and Engineering 359 (2020) 112769.

\bibitem{ji2021constructing}
Y.~Ji, Y.-Y. Yu, M.-Y. Wang, C.-G. Zhu, {Constructing high-quality planar {NURBS} parameterization for isogeometric analysis by adjustment control points and weights}, Journal of Computational and Applied Mathematics 396 (2021) 113615.

\bibitem{garanzha1999regularization}
V.~Garanzha, I.~Kaporin, {Regularization of the barrier variational method}, Computational Mathematics and Mathematical Physics 39~(9) (1999) 1426--1440.

\bibitem{garanzha2021foldover}
V.~Garanzha, I.~Kaporin, L.~Kudryavtseva, F.~Protais, N.~Ray, D.~Sokolov, Foldover-free maps in 50 lines of code, ACM Transactions on Graphics (TOG) 40~(4) (2021) 1--16.

\bibitem{wang2021smooth}
X.~Wang, W.~Ma, {Smooth analysis-suitable parameterization based on a weighted and modified Liao functional}, Computer-Aided Design 140 (2021) 103079.

\bibitem{ji2022penalty}
Y.~Ji, M.-Y. Wang, M.-D. Pan, Y.~Zhang, C.-G. Zhu, Penalty function-based volumetric parameterization method for isogeometric analysis, Computer Aided Geometric Design 94 (2022) 102081.

\bibitem{nian2016planar}
X.~Nian, F.~Chen, {Planar domain parameterization for isogeometric analysis based on Teichm{\"u}ller mapping}, Computer Methods in Applied Mechanics and Engineering 311 (2016) 41--55.

\bibitem{pan2018low}
M.~Pan, F.~Chen, W.~Tong, {Low-rank parameterization of planar domains for isogeometric analysis}, Computer Aided Geometric Design 63 (2018) 1--16.

\bibitem{martin2009volumetric}
T.~Martin, E.~Cohen, R.~M. Kirby, {Volumetric parameterization and trivariate B-spline fitting using harmonic functions}, Computer Aided Geometric Design 26~(6) (2009) 648--664.

\bibitem{nguyen2010parameterization}
T.~Nguyen, B.~J{\"u}ttler, Parameterization of contractible domains using sequences of harmonic maps, in: International Conference on Curves and Surfaces, Springer, 2010, pp. 501--514.

\bibitem{xu2013constructing}
G.~Xu, B.~Mourrain, R.~Duvigneau, A.~Galligo, {Constructing analysis-suitable parameterization of computational domain from CAD boundary by variational harmonic method}, Journal of Computational Physics 252 (2013) 275--289.

\bibitem{falini2015planar}
A.~Falini, J.~{\v{S}}peh, B.~J{\"u}ttler, {Planar domain parameterization with THB-splines}, Computer Aided Geometric Design 35 (2015) 95--108.

\bibitem{hinz2020pde}
J.~Hinz, {PDE-based parameterization techniques for isogeometric analysis applications}, Ph.D. thesis, Delft University of Technology (2020).

\bibitem{zhang2012solid}
Y.~Zhang, W.~Wang, T.~J. Hughes, {Solid T-spline construction from boundary representations for genus-zero geometry}, Computer Methods in Applied Mechanics and Engineering 249--252 (2012) 185--197.

\bibitem{zhang2013conformal}
Y.~Zhang, W.~Wang, T.~J. Hughes, {Conformal solid T-spline construction from boundary T-spline representations}, Computational Mechanics 51~(6) (2013) 1051--1059.

\bibitem{liu2014volumetric}
L.~Liu, Y.~Zhang, T.~J. Hughes, M.~A. Scott, T.~W. Sederberg, {Volumetric T-spline construction using Boolean operations}, Engineering with Computers 30 (2014) 425--439.

\bibitem{xu2019efficient}
G.~Xu, B.~Li, L.~Shu, L.~Chen, J.~Xu, T.~Khajah, {Efficient r-adaptive isogeometric analysis with Winslow's mapping and monitor function approach}, Journal of Computational and Applied Mathematics 351 (2019) 186--197.

\bibitem{ji2022curvature}
Y.~Ji, M.-Y. Wang, Y.~Wang, C.-G. Zhu, Curvature-based r-adaptive planar {NURBS} parameterization method for isogeometric analysis using bi-level approach, Computer-Aided Design 150 (2022) 103305.

\bibitem{van2011survey}
O.~Van~Kaick, H.~Zhang, G.~Hamarneh, D.~Cohen-Or, A survey on shape correspondence, Computer graphics forum 30~(6) (2011) 1681--1707.

\bibitem{sahilliouglu2020recent}
Y.~Sahillio{\u{g}}lu, Recent advances in shape correspondence, The Visual Computer 36~(8) (2020) 1705--1721.

\bibitem{zheng2021volumetric}
Y.~Zheng, F.~Chen, Volumetric boundary correspondence for isogeometric analysis based on unbalanced optimal transport, Computer-Aided Design 140 (2021) 103078.

\bibitem{zhan2023boundary}
Z.~Zhan, Y.~Zheng, W.~Wang, F.~Chen, {Boundary correspondence for isogeometric analysis based on deep learning}, Communications in Mathematics and Statistics 11~(1) (2023) 131--150.

\bibitem{lopez2022parallel}
H.~Lopez-Menchon, E.~Ubeda, A.~Heldring, J.~M. Rius, {A parallel Monte Carlo method for solving electromagnetic scattering in clusters of dielectric objects}, Journal of Computational Physics 463 (2022) 111231.

\bibitem{trefethen1980numerical}
L.~N. Trefethen, {Numerical computation of the Schwarz--Christoffel transformation}, SIAM Journal on Scientific and Statistical Computing 1~(1) (1980) 82--102.

\bibitem{driscoll1998numerical}
T.~A. Driscoll, S.~A. Vavasis, {Numerical conformal mapping using cross-ratios and Delaunay triangulation}, SIAM Journal on Scientific Computing 19~(6) (1998) 1783--1803.

\bibitem{delillo2011numerical}
T.~K. Delillo, E.~H. Kropf, {Numerical computation of the Schwarz--Christoffel transformation for multiply connected domains}, SIAM Journal on Scientific Computing 33~(3) (2011) 1369--1394.

\bibitem{Trefethoen1983SCPACK}
L.~N. Trefethen, {SCPACK: A FORTRAN77 package for Schwarz-Christoffel conformal mapping}, \url{https://www.netlib.org/conformal/} (1983).

\bibitem{driscoll1996algorithm}
T.~A. Driscoll, {Algorithm 756: A MATLAB toolbox for Schwarz-Christoffel mapping}, ACM Transactions on Mathematical Software (TOMS) 22~(2) (1996) 168--186.

\bibitem{driscoll2005algorithm}
T.~A. Driscoll, {Algorithm 843: improvements to the Schwarz-Christoffel toolbox for MATLAB}, ACM Transactions on Mathematical Software (TOMS) 31~(2) (2005) 239--251.

\bibitem{banjai2003multipole}
L.~Banjai, L.~N. Trefethen, {A multipole method for Schwarz--Christoffel mapping of polygons with thousands of sides}, SIAM Journal on Scientific Computing 25~(3) (2003) 1042--1065.

\bibitem{andersson2008schwarz}
A.~Andersson, {Schwarz--Christoffel mappings for nonpolygonal regions}, SIAM Journal on Scientific Computing 31~(1) (2008) 94--111.

\bibitem{driscoll2002schwarz}
T.~A. Driscoll, L.~N. Trefethen, {Schwarz--Christoffel mapping}, Vol.~8, Cambridge university press, 2002.

\bibitem{howell1990modified}
L.~H. Howell, L.~N. Trefethen, {A modified Schwarz--Christoffel transformation for elongated regions}, SIAM journal on scientific and statistical computing 11~(5) (1990) 928--949.

\bibitem{juttler2014geometry}
B.~J{\"u}ttler, U.~Langer, A.~Mantzaflaris, S.~E. Moore, W.~Zulehner, {Geometry + simulation modules: Implementing isogeometric analysis}, PAMM 14~(1) (2014) 961--962.

\bibitem{mantzaflaris2019overview}
A.~Mantzaflaris, An overview of geometry plus simulation modules, in: International Conference on Mathematical Aspects of Computer and Information Sciences, Springer, 2019, pp. 453--456.

\bibitem{eigenweb}
G.~Guennebaud, B.~Jacob, et~al., Eigen v3, http://eigen.tuxfamily.org (2010).

\bibitem{zhan2024simultaneous}
Z.~Zhan, W.~Wang, F.~Chen, Simultaneous boundary and interior parameterization of planar domains via deep learning, Computer-Aided Design 166 (2024) 103621.

\bibitem{pan2019low}
M.~Pan, F.~Chen, {Low-rank parameterization of volumetric domains for isogeometric analysis}, Computer-Aided Design 114 (2019) 82--90.

\bibitem{piegl1996nurbs}
L.~Piegl, W.~Tiller, {The NURBS book}, Springer Science \& Business Media, 1996.

\end{thebibliography}

\end{document}